\documentclass[twocolumn]{autart}    

\usepackage{graphicx}          
\usepackage{amssymb}
\usepackage[cmex10]{amsmath}

\newtheorem{theorem}{Theorem}
\newtheorem{coro}{Corollary}
\newtheorem{remark}{Remark}
\newtheorem{lemma}{Lemma}
\newtheorem{assumption}{Assumption}
\newtheorem{definition}{Definition}
\newtheorem{problem}{Problem}
\newtheorem{proof}{Proof}

\begin{document}

\begin{frontmatter}

\title{Control for Networked Control Systems with Remote and Local Controllers over Unreliable Communication Channel\thanksref{footnoteinfo}} 

\thanks[footnoteinfo]{\textbf{This work has been submitted to Automatica on July 25, 2017 and this is the first version for review}. This work is the first version. This work is supported by the Taishan Scholar Construction Engineering by Shandong Government, the National Natural Science Foundation of China (61120106011, 61633014, 61403235, 61573221). Corresponding author Juanjuan Xu.}

\author[China]{Xiao Liang}\ead{liangxiao\_sdu@163.com},    
\author[China]{Juanjuan Xu}\ead{jnxujuanjuan@163.com}               

\address[China]{School of Control Science and Engineering, Shandong University, Jinan, Shandong, P.R.China}  


\begin{abstract}                          
This paper is concerned with the problems of optimal control and stabilization for networked control systems (NCSs), where the remote controller and the local controller operate the linear plant simultaneously. The main contributions are two-fold. Firstly, a necessary and sufficient condition for the finite horizon optimal control problem is given in terms of the two Riccati equations. Secondly, it is shown that the system without the additive noise is stabilizable in the mean square sense if and only if the two algebraic Riccati equations admit the unique solutions, and a sufficient condition is given for the boundedness in the mean square sense of the system with the additive noise. Numerical examples about the unmanned aerial vehicle model are shown to illustrate the effectiveness of the proposed algorithm.

\end{abstract}

\end{frontmatter}

\section{Introduction}
Networked control systems (NCSs) are control systems consisting of controllers, sensors and actuators which are spatially distributed and coordinated via a certain digital communication networks \cite{R1}-\cite{R5}. Recently, NCSs have received considerable interest due to their applications in different areas, such as automated highway systems, unmanned aerial vehicles and large manufacturing systems \cite{R6}-\cite{R10}. Comparing with the classical feedback control systems, NCSs have vast superiority including low cost, reduced power requirements, simple maintenance and high reliability. However, the packet dropout occurred in the communication channels of NCSs brings in challenging problems \cite{R11}-\cite{R14}. Therefore, it is of great significance to study NCSs with unreliable communication channels where the packet dropout happens.

The research on the packet dropout have been studied in \cite{R15}-\cite{R22}. \cite{R16} introduced the Kalman filter with intermittent observations and the optimal estimator is defined as $\hat{x}_{k|k}=E[x_k|\{z_k\},\{\gamma_k\}]$ with conditioning on the arrival process $\{\gamma_k\}$. \cite{R18} derived the optimal estimator without conditioning on the arrival process $\{\gamma_k\}$ and obtained the optimal controller for the systems subject to the state packet dropout. In \cite{R20}, the optimal linear quadratic Gaussian control for system involving the packet dropout is studied by decomposing the problem into a standard LQR state-feedback controller design, along with an optimal encoder-decoder design. The stabilization problem is investigated in \cite{R22} for NCSs with the packet dropout. Nevertheless, the aforementioned literatures are merely involved in a single controller.

Inspired by the previous work \cite{R38}, the NCSs under consideration of this paper are depicted as in Fig.~\ref{fig1}, which is composed of a plant, a local controller, a remote controller and an unreliable communication channel. The state $x_k$ can be perfectly observed by the local controller. Then, the state $x_k$ is sent to the remote controller via an unreliable communication channel where the packet dropout occurs with probability $p$. We define $y_k$ as the observed signal of the remote controller. When the remote controller observes the signal $y_k$, an acknowledgement is sent from the remote controller to the local controller whether the state is received. Hence, the local controller can observe the signal $y_k$ as well. The two controllers will not perform their control actions until observe the signal $y_k$. \emph{It is stressed that the remote control $u^R_k$ is not available for the local control $u^L_k$ at the same time k.} Besides, the channels from the controllers to the plant are assumed to be perfect. The aim of the optimal control is to minimize the quadratic performance cost of NCSs and stabilize the linear plant.
\begin{figure}[htbp]
  \begin{center}
  \includegraphics[width=0.42\textwidth]{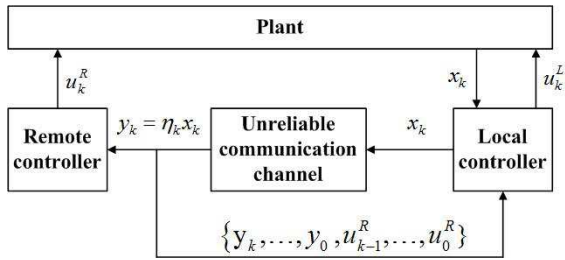}
  \caption{Overview of NCSs with an unreliable communication channel.} \label{fig1}
  \end{center}
\end{figure}

The NCSs model stems from increasing applications that appeal for remote control of objects over Internet-type or wireless networks where the communication channels are prone to failure. The local controller can be an integrated chip on the unmanned aerial vehicle (UAV) that implements moderate control and is poor in the transmission capability, while the remote controller can be a ground-control center which is equipped with complete communication installation and is capable of powerful transmission. Therefore, the communication channel from the local controller to the remote controller is prone to failure. Inversely, the transmission channel from the remote controller to the local controller is normally substantial.

For two decision-makers, the general control strategies are Nash equilibrium and Stackelberg strategy; see \cite{R23}-\cite{R30}. The relationship between Nash equilibrium strategies and the finite horizon $H_2/H_\infty$ control of time-varying stochastic systems subject to Markov jump parameters and multiplicative noise has been studied in \cite{R23}. \cite{R24} investigated the necessary conditions for the Nash game in terms of two coupled and nonsymmetric Riccati equations. \cite{R27} investigated the properties of the Stackelberg solution in static and dynamic nonzero-sum two-play games. In \cite{R29}, the optimal open-loop Stackelberg strategy was designed for a two-player game in terms of three decoupled and symmetric Riccati equations. For the Nash equilibrium, it is necessary for each controllers to access the optimal control strategies of each other, which is different from the idea of this paper. Although the Stackelberg strategy is an available method for two decision-makers, it is assumed that one player is capable of announcing his strategy before the other, which is not applicable in this work. Due to the asymmetric information for the remote controller and the local controller, the analysis and synthesis for the optimal control remain challenging. What's more, one controller can not obtain the current action of the other controller which makes the optimal control problem more difficult. More recently, \cite{R38} studied the similar NCSs model as in this paper. However, in \cite{R38}, it is noted that only sufficient condition was given for the optimal control and the stabilization problem was missing which is of great significance in applications.

In this paper, we shall study the optimal control and stabilization problems for the NCSs with remote and local controllers over unreliable communication channel. The key technique is to apply the Pontryagin's maximum principle to develop a direct approach based on the solution to the forward backward stochastic difference equations (FBSDEs), which will lead to a non-homogeneous relationship between the state estimation and the costate. The main contributions of this papers are summarized as follows: (1) An explicit solution to the FBSDEs is presented with the Pontryagin's maximum principle. Using this solution, a necessary and sufficient condition for the finite horizon optimal control problem is given in terms of the solutions to the two Riccati equations. (2) For the stochastic systems without the additive noise, a necessary and sufficient condition for stabilizing the systems in the mean-square sense is developed. For the stochastic systems with the additive noise, a sufficient condition is derived for the boundedness in the mean-square sense of the systems.

The rest of the paper is organized as follows. The finite horizon optimal control problem is studied in Section II. In Section III, the infinite horizon optimal control and the stabilization problem are solved. Numerical examples about the unmanned aerial vehicle are given in Section IV. The conclusions are provided in Section V. Relevant proofs are detailed in Appendices.

\emph{Notation:} $R^n$ denotes the $n$-dimensional Euclidean space. $I$ presents the unit matrix of appropriate dimension. $A'$ denotes the transpose of the matrix $A$. $\mathcal{F}(X)$ denotes the $\sigma$-algebra generated by the random variable $X$. $A\geq0 (>0)$ means that $A$ is a positive semi-definite (positive definite) matrix. Denote $E$ as mathematical expectation operator. $Tr(A)$ represents the trace of matrix $A$.


\section{Finite Horizon Case}
\subsection{Problem Formulation}
Consider the discrete-time system with two controllers as shown in Fig.1. The corresponding linear plant is given by
\begin{align}
x_{k+1}=Ax_k+B^Lu^L_k+B^Ru^R_k+\omega_k,\label{1}
\end{align}
where $x_k\in R^{n_x}$ is the state, $u^L_k\in R^{n_L}$ is the local control, $u^R_k\in R^{n_R}$ is the remote control, $\omega_k$ is the input noise and $A$, $B^L$, $B^R$ are constant matrices with appropriate dimensions. The initial state $x_0$ and $\omega_k$ are Gaussian and independent, with mean $(\bar{x}_0,0)$ and covariance $(\bar{P}_0,Q_{\omega_k})$ respectively.

As can be seen in Fig.1, let $\eta_k$ be an independent identically distributed (i.i.d) Bernoulli random variable describing the state signal transmitted through the unreliable communication channel, i.e., $\eta_k=1$ denotes that the state packet has been successfully delivered, and $\eta_k=0$ signifies the dropout of the state packet. Then,
\begin{equation}
\eta_k=\left\{\begin{array}{ccc}
1,&\textrm{$y_{k}=x_k$, with probability $(1-p)$} \\
0,&\textrm{$y_{k}=\emptyset$, with probability $p$}\end{array}\right.
\end{equation}

Observing Fig.1, the remote control $u^R_k$ can obtain the signals $\{y_0,\ldots,y_k\}$. Accordingly, we have that $u^R_k$ is $\mathcal{F}\{y_0,\ldots,y_k\}$-measurable. The local control $u^L_k$ has access to the states $\{x_0,\ldots,x_k\}$ and the signals $\{y_0,\ldots,y_k\}$.
In view of (1), we have that $u^L_k$ is $\mathcal{F}\{x_0,\omega_0,\ldots,\omega_{k-1},y_0,\ldots,y_k\}$-measurable. For simplicity, we denote $\mathcal{F}\{x_0,\omega_0,\ldots,\omega_{k-1},y_0,\ldots,y_k\}$ and $\mathcal{F}\{y_0,\ldots,y_k\}$ as $\mathcal{F}_k$ and $\mathcal{F}{\{Y_k\}}$ respectively.

The associated performance index for system (\ref{1}) is given by
\begin{align}
\nonumber J_N=&E\bigg\{\sum_{k=0}^N[{x_k}'Qx_k+(u_{k}^L)'R^Lu_{k}^L+(u_{k}^R)'R^Ru_{k}^R]\\
             &\quad\quad+{x_{N+1}}'P_{N+1}x_{N+1}\bigg\}\label{2}
\end{align}
where $Q$, $R^L$, $R^R$ and $P_{N+1}$ are positive semi-definite matrices. $E$ takes the mathematical expectation over the random process $\{\eta_k\}$, $\{\omega_k\}$ and the random variable $x_0$.

Thus, the optimal control strategies to be addressed are stated as follows:
\begin{problem}
Find the $\mathcal{F}_k$-measurable $u^L_k$ and the $\mathcal{F}{\{Y_k\}}$-measurable $u^R_k$ such that (\ref{2}) is minimized, subject to (\ref{1}).
\end{problem}
\subsection{Optimal Controllers Design}
Following the results in \cite{R31}, we apply Pontryagin's maximum principle to the system (\ref{1}) with the cost function (\ref{2}) to yield the following costate equations:
\begin{align}
\lambda_{k-1}&=E\left[A'\lambda_k|\mathcal{F}_k\right]+Qx_k, k=0,\ldots,N,\label{12}\\
0&=E\left[(B^L)'\hspace{-0.8mm}\lambda_k|\mathcal{F}_k\right]\hspace{-0.8mm}+\hspace{-0.8mm}R^Lu^L_k,\label{13}\\
0&=E\left[(B^R)'\lambda_k|\mathcal{F}\{Y_k\}\right]+R^Ru^R_{k},\label{6}\\
\lambda_N&=P_{N+1}x_{N+1},\label{5}
\end{align}
where $\lambda_k$ is the costate.

It is noted that the key to obtain the optimal controllers is to solve the costate equations (\ref{12})-(\ref{5}). To this end, we define the following two Riccati equations:
\begin{align}
Z_k&=A'Z_{k+1}A+Q-K_k'\Upsilon_kK_k,\label{9}\\
X_k&=A'\Psi_{k+1}A+Q-M_k'\Lambda_k^{-1}M_k,\label{19}
\end{align}
where
\begin{align}
 K_k&=\Upsilon_k^{-1}\begin{bmatrix}B^L&B^R\end{bmatrix}'Z_{k+1}A,\label{27}\\
 \Upsilon_k&=\begin{bmatrix}B^L&B^R\end{bmatrix}'Z_{k+1}\begin{bmatrix}B^L&B^R\end{bmatrix}+\begin{bmatrix}R^L&0\\0&R^R\end{bmatrix},\label{28}\\
\Psi_k&=(1-p)Z_{k}+pX_{k},\label{67}\\
\Lambda_k&=(B^L)'\Psi_{k+1}B^L+R^L,\label{29}\\
M_k&=(B^L)'\Psi_{k+1}A,\label{30}
\end{align}
with terminal values $Z_{N+1}=X_{N+1}=P_{N+1}$.

It can be observed that the equation (\ref{9}) of $Z_k$ is a standard Riccati difference equation. By using the solutions to these two Riccati equations, we firstly propose the following lemma.
\begin{lemma}
Assuming that $\Upsilon_k>0$ and $\Lambda_k>0$ for $k=0,\ldots,N+1$, then the following equation
\begin{align}
\lambda_{k-1}=Z_k\hat{x}_{k|k}+X_k\tilde{x}_{k},\label{31}
\end{align}
is the solution to the FBSDEs
\begin{align}
\nonumber x_{k+1}&=Ax_k+B^Lu^L_k+B^Ru^R_k+\omega_k,\\
\nonumber \lambda_{k-1}&=E\left[A'\lambda_k|\mathcal{F}_k\right]+Qx_k, k=0,\ldots,N,\\
\nonumber \lambda_N&=P_{N+1}x_{N+1},
\end{align}
with $u^L_k$ and $u^R_{k}$ satisfy
\begin{align}
\nonumber0&=E\left[(B^L)'\hspace{-0.8mm}\lambda_k|\mathcal{F}_k\right]\hspace{-0.8mm}+\hspace{-0.8mm}R^Lu^L_k,\\
\nonumber0&=E\left[(B^R)'\lambda_k|\mathcal{F}\{Y_k\}\right]+R^Ru^R_{k},
\end{align}
where $\hat{x}_{k|k}$ and $\tilde{x}_k$ are the estimation and the estimation error of the state $x_k$ respectively,
\begin{align}
\hat{x}_{k|k}&=E[x_k|\mathcal{F}\{Y_k\}]=(1-\eta_k)\hat{x}_{k|k-1}+\eta_kx_k,\label{80}\\
          \tilde{x}_{k}&=x_k-\hat{x}_{k|k},\label{22}
\end{align}
with initial value
\begin{align}
\hat{x}_{0|0}&=(1-\eta_0)\bar{x}_0+\eta_0x_0,\label{83}\\
\tilde{x}_{0}&=x_0-[(1-\eta_0)\bar{x}_0+\eta_0x_0]\label{84},
\end{align}
and $Z_k$ and $X_k$ are as (\ref{9}) and (\ref{19}).
\end{lemma}
\begin{proof}
The proof is put into Appendix A.
\end{proof}
The optimal control strategies for $u^R_k$ and $u^L_k$ are given in the theorem below.
\begin{theorem}
Problem 1 has the unique solution if and only if $\Upsilon_k>0$ and $\Lambda_k>0$ as in (\ref{28}) and (\ref{29}) for $k=N,\ldots,0$.

In this case, the optimal controls $u^R_k$ and $u^L_k$ are given by
\begin{align}
u^R_k&=-\begin{bmatrix}0&I\end{bmatrix}K_k\hat{x}_{k|k},\label{8}\\
u^L_k&=-\begin{bmatrix}I&0\end{bmatrix}K_k\hat{x}_{k|k}-\Lambda_k^{-1}M_k\tilde{x}_{k}.\label{14}
\end{align}
where $K_k$, $\Lambda_k$, and $M_k$ are as (\ref{27}), (\ref{29}) and (\ref{30}). $\hat{x}_{k|k}$ and $\tilde{x}_k$ are defined as in Lemma 1. The associated optimal cost is given by
\begin{align}
\nonumber J^*_N&=E\big[x_0'Z_0\hat{x}_{0|0}+x_0'X_0\tilde{x}_0\big]+\sum_{k=0}^{N}\big[p Tr(X_{k+1}Q_{\omega_k})\\
               &\quad+\hspace{-0.8mm}(1-p)Tr(Z_{k+1}Q_{\omega_k})\big].\label{25}
\end{align}
\end{theorem}
\begin{proof}
The proof of is put into Appendix B.
\end{proof}
\begin{remark}
Based on the results in this paper, the following points are highlighted to make comparison with \cite{R38}. (1) We study both the finite-horizon optimal control and the stabilization problems for the NCSs with remote and local controllers, while the stabilization problem is not considered in \cite{R38}. (2) The key technique herein is the Pontryagin's maximum principle, which is different from the dynamic programming adopted in \cite{R38}. (3) The results obtained in this paper are necessary and sufficient which solve the problem completely, while only sufficient condition is given in \cite{R38}. (4) More specific, the weighting matrices in the finite-horizon cost function are positive semi-definite which are more general than the positive definite ones in \cite{R38}.
\end{remark}
\section{Infinite Horizon Case}
\subsection{Problem Formulation}
In this section, we will solve the infinite horizon optimal control and stabilization problem.

Since the additive noise is present, only a sufficient condition for the boundedness in the mean square sense of the system can be derived; see \cite{R32} and \cite{R33}. In other words, system (\ref{1}) cannot be stabilizable in the mean square sense due to the existence of the additive noise. In order to derive a necessary and sufficient condition for the stabilization of the system, we firstly discard the additive noise of the system (\ref{1}). The stabilization problem for the system (\ref{1}) will be discussed later.

Thus, the system (\ref{1}) without the additive noise can be written as
\begin{align}
x_{k+1}=Ax_k+B^Lu^L_k+B^Ru^R_k, \label{16}
\end{align}
where the initial value $x_0$ is Gaussian random vector with mean $\bar{x}_0$ and covariance $\bar{P}_0$.

Let $P_{N+1}=0$. We change the finite horizon cost function (\ref{2}) to the infinite horizon cost function as follows:
\begin{align} J=&E\sum_{k=0}^\infty[{x_k}'Qx_k+(u_{k}^L)'R^Lu_{k}^L+(u_{k}^R)'R^Ru_{k}^R].\label{21}
\end{align}

We start with the following definitions.
\begin{definition}
The system (\ref{16}) with $u^L_k=0$ and $u^R_k=0$ is said to be asymptotically mean-square stable if for any initial values $x_0$, the corresponding state satisfies
\begin{align}
\nonumber\lim_{k\to \infty}E({x_k'}x_k)=0.
\end{align}
\end{definition}
\begin{definition}
The system (\ref{16}) is called stabilizable in the mean-square sense if there exist the $\mathcal{F}\{Y_k\}$-measurable $u^R_k=-K_1\hat{x}_{k|k}$, and $\mathcal{F}_k$-measurable $u^L_k=-K_2\hat{x}_{k|k}-L\tilde{x}_k$, $k\geq0$ with constant matrices $K_1$, $K_2$ and $L$ such that for any $x_0\in R^n$, the closed-loop system of (\ref{16}) is asymptotically mean-square stable.
\end{definition}
From \cite{R34}, we make the following two assumptions:
\begin{assumption}
$R^L>0$, $R^R>0$ and $Q=C'C\geq0$ for some matrices $C$.
\end{assumption}
\begin{assumption}
$(A,Q^{1/2})$ is observable.
\end{assumption}
The problem to be addressed in this section is stated below:
\begin{problem}
Find the $\mathcal{F}\{Y_k\}$-measurable $u^R_k$, and $\mathcal{F}_k$-measurable $u^L_k$, $k\geq0$, such that the closed-loop system of (\ref{16}) is asymptotically mean-square stable and that (\ref{21}) is minimized.
\end{problem}
To make the time horizon $N$ explicit in the finite horizon case, we rewrite $Z_k$, $X_k$, $K_k$, $\Upsilon_k$, $\Psi_k$, $\Lambda_k$ and $M_k$ in (\ref{9})-(\ref{30}) as $Z_k(N)$, $X_k(N)$, $K_k(N)$, $\Upsilon_k(N)$, $\Psi_k(N)$, $\Lambda_k(N)$ and $M_k(N)$.
\subsection{Solution to Problem 2}
We now present the main results of this section.
\begin{theorem}
Under Assumptions 1 and 2, the system (\ref{16}) is stabilizable in the mean-square sense if and only if there exist the unique solutions $Z$ and $X$ to the following Riccati equations, satisfy $Z>0$, $\Psi>0$ and $E[{x}_{0}'Z\hat{x}_{0|0}+{x}_0'X\tilde{x}_0]\geq0$ for any initial value $x_0$:
\begin{align}
Z&=A'ZA+Q-K'\Upsilon K,\label{24}\\
X&=A'\Psi A+Q-M'\Lambda^{-1}M,\label{26}
\end{align}
where
\begin{align}
K&=\Upsilon^{-1}\begin{bmatrix}B^L&B^R\end{bmatrix}'ZA,\label{35}\\
\Upsilon&=\begin{bmatrix}B^L&B^R\end{bmatrix}'Z\begin{bmatrix}B^L&B^R\end{bmatrix}+\begin{bmatrix}R^L&0\\0&R^R\end{bmatrix},\label{60}\\
\Psi&=(1-p)Z+pX,\\
\Lambda&=(B^L)'\Psi B^L+R^L,\label{81}\\
M&=(B^L)'\Psi A.\label{36}
\end{align}
In this case, the optimal controllers
\begin{align}
u^R_k&=-\begin{bmatrix}0&I\end{bmatrix}K\hat{x}_{k|k},\label{37}\\
u^L_k&=-\begin{bmatrix}I&0\end{bmatrix}K\hat{x}_{k|k}-\Lambda^{-1}M\tilde{x}_{k},\label{38}
\end{align}
stabilize the system (\ref{16}) in the mean square sense and minimize the cost function (\ref{21}). The optimal cost is given by
\begin{align}
J^*&=E\big[x_0'Z\hat{x}_{0|0}+x_0'X\tilde{x}_0\big].\label{54}
\end{align}
\end{theorem}
\begin{proof}
The proof is put into Appendix C.
\end{proof}
Next we shall consider the stabilization problem for the system with the additive noise.
\begin{lemma}
Consider the system (\ref{1}), under Assumptions 1-2, and the system (\ref{16}) is stabilizable in the mean-square sense, then for any initial condition, the estimator error covariance matrix $\Sigma_k=E[(x_k-\hat{x}_{k|k})(x_k-\hat{x}_{k|k})']$ converges to a unique positive definite matrix if and only if $\sqrt{p}|\lambda_{max}(A-B^L\Lambda^{-1}M)|<1$, where $\lambda_{max}(A-B^L\Lambda^{-1}M)$ is the eigenvalue of matrix $(A-B^L\Lambda^{-1}M)$ with the largest absolute value, and $\Lambda,M$ satisfy (\ref{81}) and (\ref{36}).
\end{lemma}
\begin{proof}
The proof is put into Appendix D.
\end{proof}
\begin{coro}
Consider the system (\ref{1}) with the following infinite horizon cost function:
\begin{align}
\tilde{J}\hspace{-0.8mm}=\hspace{-0.8mm}\lim_{N\to\infty}\hspace{-0.8mm}\frac{1}{N}E\sum_{k=0}^N\hspace{-0.8mm}\bigg[{x_k}'Qx_k\hspace{-0.8mm}+\hspace{-0.8mm}(u^L_k)'R^Lu^L_k\hspace{-0.8mm}+\hspace{-0.8mm}(u^R_k)'R^Ru^R_k\bigg]\hspace{-0.8mm}.\label{55}
\end{align}
Under Assumptions 1 and 2, if there exist the unique solutions $Z$ and $X$ to the Riccati equations (\ref{24})-(\ref{36}) and $\sqrt{p}|\lambda_{max}(A-B^L\Lambda^{-1}M)|<1$, then the system (\ref{1}) is bounded in the mean-square sense.

In this case, the optimal controllers
\begin{align}
\check{u}^R_k&=-\begin{bmatrix}0&I\end{bmatrix}K\hat{x}_{k|k},\label{56}\\
\check{u}^L_k&=-\begin{bmatrix}I&0\end{bmatrix}K\hat{x}_{k|k}-\Lambda^{-1}M\tilde{x}_{k},\label{57}
\end{align}
make the system (\ref{1}) bounded in the mean square sense and minimize the cost function (\ref{55}). The optimal cost function is given by
\begin{align}
\tilde{J}^*&=p Tr(XQ_{\omega})+(1-p)Tr(ZQ_{\omega}),\label{58}
\end{align}
where $Q_{\omega}$ is the covariance of the additive noise $\omega_k$.
\end{coro}
\begin{proof}
The proof is put into Appendix E.
\end{proof}
\section{Numerical Examples}
The UAV systems have recently received significant attention in the controls community due to its numerous applications, including space science missions, surveillance, terrain mapping and formation flight \cite{R36} and \cite{R37}. The UAV systems have considerable advantages, such as reducing the energy cost, improving the aviation safety and so on. Besides, the UAV are used because they can outperform human pilots, remove humans from dangerous situations, and perform repetitive tasks that can be automated. In this section, we consider a simple UAV system as an example of the system model of Section II.

Consider a simple UAV system of a unmanned aerial vehicle and a ground-control center which is depicted as in Fig.~\ref{fig2}.
 \begin{figure}[htbp]
  \begin{center}
  \includegraphics[width=0.42\textwidth]{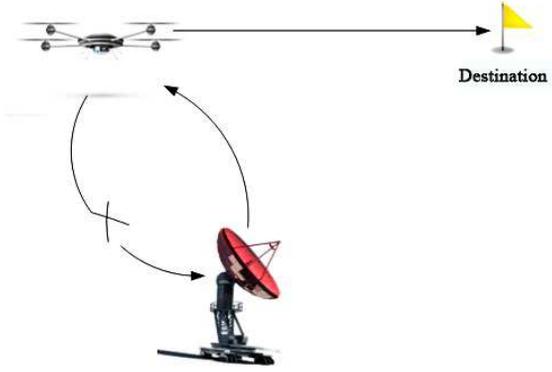}
  \caption{Overview of the UAV model.} \label{fig2}
  \end{center}
\end{figure}
Denote $\varsigma_k$ and $\upsilon_k$ as the location and the velocity of the unmanned aerial vehicle at time $k$ (for simplicity, we assume that the unmanned aerial vehicle files in the straight line). Accordingly, at time $k+1$, the location of the unmanned aerial vehicle can be written as
\begin{align}
\nonumber\varsigma_{k+1}=\varsigma_k+\upsilon_k+\omega_k,
\end{align}
where $\omega_k$ denotes the interference during the flight, e.g. wind. The initial location $\varsigma_0$ and $\omega_k$ are Gaussian and independent, with mean $(\bar{\varsigma}_0,0)$ and covariance $(\bar{P}_0,\sigma)$ respectively.

At any time $k$, the unmanned aerial vehicle can perfectly observe its location. Meanwhile, the observed location is sent to the ground-control center through an unreliable communication channel with the packet dropout probability $p$. Then, the ground-control center sends the control command to the unmanned aerial vehicle as well as the acknowledgement whether it receives the location information. The unmanned aerial vehicle makes a local decision about its velocity based on the local observation and the received information from the ground-control center. Due to the sufficient equipment of the ground-control center, the communication channel from the ground-control center to the unmanned aerial vehicle is perfect.

The aim of the UAV system is to make the unmanned aerial vehicle reach the Destination $\varsigma$ while the energy cost is minimum. Thus, we denote the above aim by a cost function
\begin{align}
J_N=\sum_{k=0}^NE\left[(\varsigma_k-\varsigma)'Q(\varsigma_k-\varsigma)+\upsilon_k'R\upsilon_k\right],\label{69}
\end{align}
where the first term is the sum of quadratic distance between the real-time location and the Destination $\varsigma$, the second term is the sum of the quadratic real-time velocity, $Q\geq0$ and $R\geq0$ are the weighting coefficients.

This UAV system can be described by applying the NCS model in Section II. Define $x_{k}=\varsigma_k-\varsigma$, $\upsilon^L_k+\upsilon^R_k=\upsilon_k$. Then, the corresponding linear plant is given by
\begin{align}
x_{k+1}=x_{k}+\upsilon^L_k+\upsilon^R_k+\omega_k.\label{70}
\end{align}
Accordingly, ignoring the cross terms, (\ref{69}) can be written as
\begin{align}
\hspace{-0.8mm}J_N\hspace{-0.8mm}=&E\sum_{k=0}^N\hspace{-0.8mm}\bigg[{x_k}'Qx_k\hspace{-0.8mm}+\hspace{-0.8mm}(\upsilon^R_k)'R\upsilon^R_k\hspace{-0.8mm}+\hspace{-0.8mm}(\upsilon_{k}^L)'R\upsilon_{k}^L\bigg],\label{71}
\end{align}
\subsection{Finite horizon case}
By applying Theorem 1, the optimal strategies are derived as follows:
\begin{align}
\nonumber\upsilon^R_k&=-\hspace{-0.8mm}\begin{bmatrix}0&I\end{bmatrix}K_k\hat{x}_{k|k},\\
\nonumber\upsilon^L_k&=-\hspace{-0.8mm}\begin{bmatrix}I&0\end{bmatrix}K_k\hat{x}_{k|k}\hspace{-0.8mm}-\hspace{-0.8mm}\Lambda_k^{-1}M_k\tilde{x}_{k},\\
\nonumber\upsilon_k&=\upsilon^L_k+\upsilon^R_k,
\end{align}
where
\begin{align}
\nonumber\hat{x}_{k|k}&=(1-\eta_k)\hat{x}_{k|k-1}+\eta_kx_k,\\
\nonumber\tilde{x}_{k}&=x_k-\hat{x}_{k|k},
\end{align}
with initial value
\begin{align}
\nonumber \hat{x}_{0|0}&=(1-\eta_0)\bar{x}_0+\eta_0x_0, \\
\nonumber\tilde{x}_{0}&=x_0-[(1-\eta_0)\bar{x}_0+\eta_0x_0].
\end{align}
The optimal cost is as
\begin{align}
\nonumber J^*_N&=E\big[x_0'Z_0\hat{x}_{0|0}+x_0'X_0\tilde{x}_0\big]+\sum_{k=0}^{N}\big[p Tr(X_{k+1}\sigma)\\
               &\quad+(1-p)Tr(Z_{k+1}\sigma)\big].
\end{align}
We consider the system (\ref{70}) with $\bar{\varsigma}_0=0, \varsigma=30, \bar{P}_0=1, \sigma=1$, and (\ref{71}) with $R=5, Q=0.01, P_{N+1}=0, N=100$.

Then, by applying Corollary 2, Fig.~\ref{fig3} and Fig.~\ref{fig4} are portrayed. Fig.~\ref{fig3} indicates the velocity $\upsilon_k$ of the unmanned aerial vehicle for packet dropout probability $p=0, p=0.5, p=1$. It can be seen that there is no huge difference in the velocity of the unmanned aerial vehicle for different values of packet dropout probability $p$. Fig.~\ref{fig4} shows the optimal energy cost $J^*_N$ for different values of $p$. Obviously, the energy cost increases greatly with the packet dropout probability.
\subsection{Infinite horizon case}
The stabilization performance of the UAV system is to be shown. Firstly, we will consider the case without the additive noise $\omega_k$.

Let $p=0.5$ and other variables have the same values as in the finite horizon case. Using Theorem 2, Fig.~\ref{fig5} is drawn as the dynamic behavior of $E[x_k'x_k]$. It can be seen that the regulated state is mean-square stable.

Now, we shall deal with the case with the additive noise $\omega_k$. Let $p=0.6$ and other variables have the same values as in the finite horizon case. By applying Corollary 1, the dynamic behavior of $E[x_k'x_k]$ is presented in Fig.~\ref{fig6} which indicates that the regulated state is mean-square bounded.
\begin{figure}[htbp]
  \begin{center}
  \includegraphics[width=0.42\textwidth]{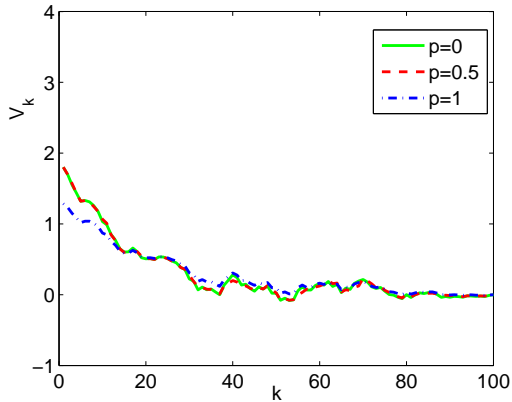}
  \caption{Velocity of the unmanned aerial vehicle for $p=0, p=0.5, p=1$.} \label{fig3}
  \end{center}
\end{figure}
\begin{figure}[htbp]
  \begin{center}
  \includegraphics[width=0.42\textwidth]{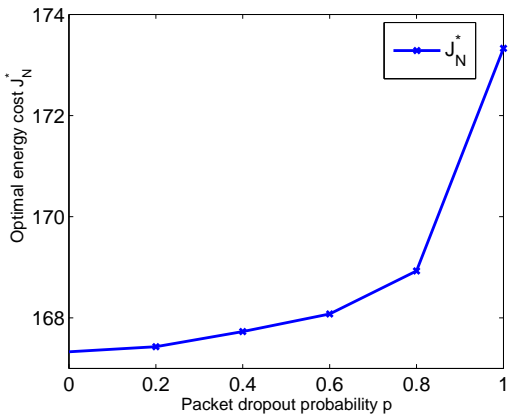}
  \caption{Optimal energy cost for different $p$.} \label{fig4}
  \end{center}
\end{figure}
\begin{figure}[htbp]
  \begin{center}
  \includegraphics[width=0.42\textwidth]{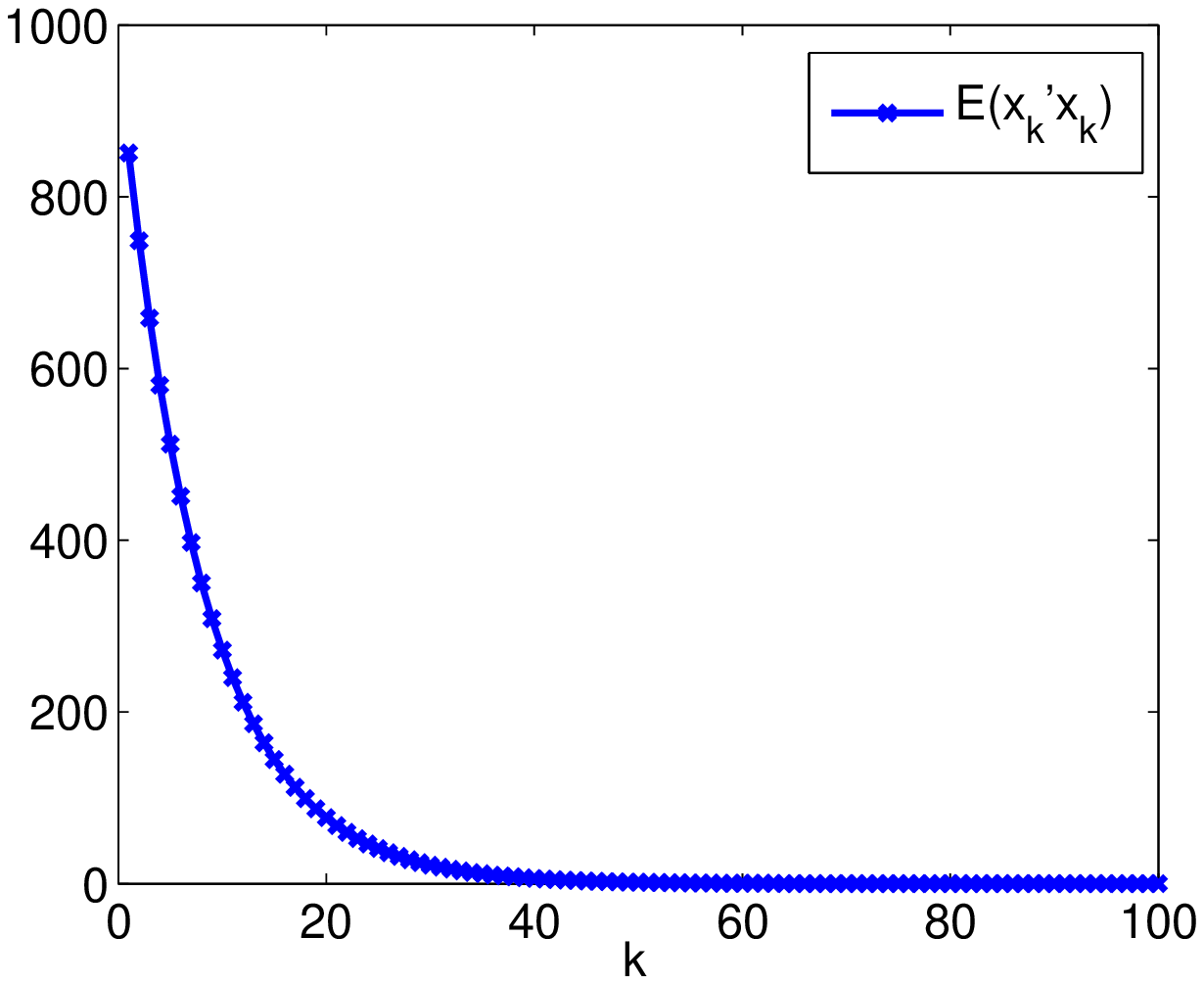}
  \caption{Dynamic Behavior of $E[x_k'x_k]$} \label{fig5}
  \end{center}
\end{figure}
\begin{figure}[htbp]
  \begin{center}
  \includegraphics[width=0.42\textwidth]{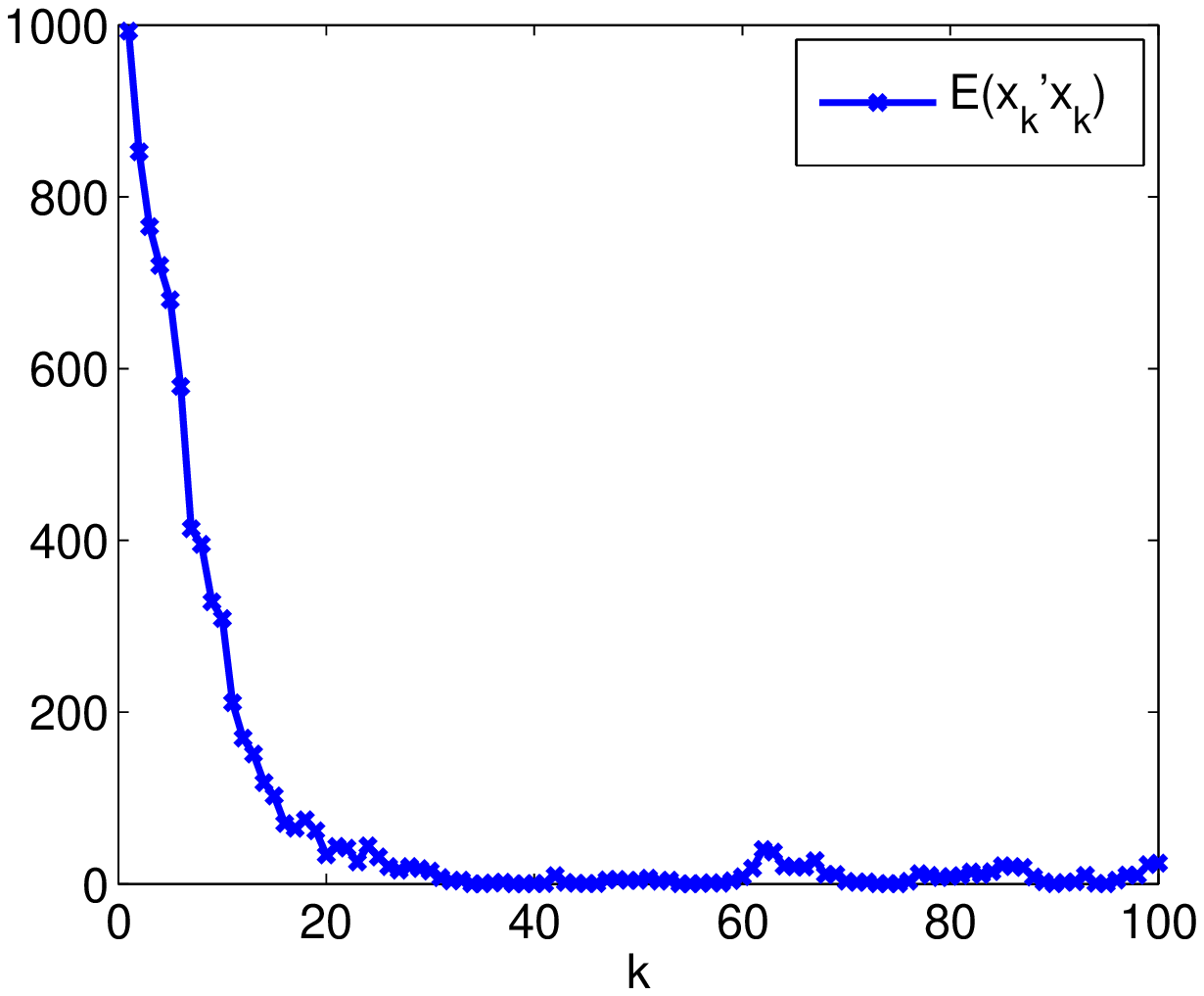}
  \caption{Dynamic Behavior of $E[x_k'x_k]$} \label{fig6}
  \end{center}
\end{figure}
\section{Conclusion}
In this paper, the optimal control and stabilization problems for NCSs have been studied. In NCSs, the linear plant is controlled by the remote controller and the local controller. The local controller perfectly observes the state signal and sends the state signal to the remote controller with packet dropout. Then, the remote controller sends an acknowledge to the local controller. It is stressed that remote control $u^R_k$ is not available for the local control $u^L_k$ at time $k$. By applying the Pontryagin's maximum principle, a non-homogeneous relationship between the state estimation and the costate is developed. Based on this relationship, a necessary and sufficient condition for the finite horizon optimal control problem is given in terms of the solutions to the Riccati equations. For the infinite horizon case, a necessary and sufficient condition for stabilizing the systems without the additive noise in the mean-square sense is developed. For the systems with the additive noise, a sufficient condition is derived for the boundedness in the mean-square sense of the systems. Furthermore, we apply the obtained results to a simple UAV system which shows the effectiveness of this algorithm.

\begin{ack}                               
The authors would like to thank Prof. Huanshui Zhang for his valuable discussions.
\end{ack}

\bibliographystyle{plain}        
\bibliography{autosam}           

\begin{thebibliography}{99}     
\bibitem[Zhang {\em et al.}, 2001]{R1}
W. Zhang, M. S. Branicky, and S. M. Phillips (2001), Stability of networked control systems. {\it IEEE Control Systems Magazine} vol. 21, no. 1, (84--99).
\bibitem[Schenato {\em et al.}, 2007]{R2}
L. Schenato, B. Sinopoli, M. Franceschetti, K. Poolla, and S. Sastry (2007), Foundations of control and estimation over lossy networks. {\it Proceedings of the IEEE} vol. 95, no. 1, (163--187).
\bibitem[Hu {\em et al.}, 2003]{R3}
S. Hu and W. Zhu (2003), Stochastic optimal control and analysis of stability of networked control systems with long delay. {\it Automatica} vol. 39, (1877--1884).
\bibitem[Yang {\em et al.}, 2011]{R4}
R. Yang, P. Shi, G. P. Liu, and H. Gao (2011), Network-based feedback control for systems with mixed delays based on quantization and dropout compensation. {\it Automatica} vol. 47, no. 12, (2805--2809).
\bibitem[Hespanha {\em et al.}, 2007]{R5}
J. P. Hespanha, P. Naghshtabrizi, and Y. Xu (2007), A survey of recent results in networked control systems. {\it Proceedings of the IEEE} vol. 95, no. 1, (138--162).
\bibitem[Horowitz {\em et al.}, 2007]{R6}
R. Horowitz and P. Varaiya (2000), Control design of an automated highway system. {\it Proceedings of the IEEE} vol. 88, no. 7, (913--925).
\bibitem[Seiler, 2001]{R7}
P. J. Seiler (2001), Coordinated control of unmanned aerial vehicles. {\it PhD thesis} University of California, Berkeley.
\bibitem[Zhang {\em et al.}, 2013]{R8}
L. Zhang, H. Gao, and O. Kaynak (2013), Network-induced constraints in networked control systems¡ªA survey. {\it IEEE Trans. Ind. Informat.} vol. 9, no. 1, (403--416).
\bibitem[He {\em et al.}, 2009]{R9}
X. He, Z. Wang, and D. Zhou (2009), Robust fault detection for networked systems with communication delay and data missing. {\it Automatica} vol. 45, no. 11, (2634--2639).
\bibitem[Faezipour {\em et al.}, 2012]{R10}
M. Faezipour, M. Nourani, A. Saeed, and S. Addepalli (2012), Progress and challenges in intelligent vehicle area networks. {\it Commun. ACM} vol. 55, no. 2, (90--100).
\bibitem[Wu {\em et al.}, 2007]{R11}
J. Wu and T. Chen (2007), Design of networked control systems with packet dropouts. {\it IEEE Trans. Autom. Control} vol. 52, no. 7, (1314--1319).
\bibitem[Pang {\em et al.}, 2016]{R12}
Z. H. Pang, G. P. Liu, D. Zhou, and D. Sun (2016), Data-based predictive control for networked nonlinear systems with network-induced delay and packet dropout. {\it IEEE Trans. Ind. Electron.} vol. 63, no. 2, (1249--1257).
\bibitem[Sun {\em et al.}, 2016]{R13}
S. L. Sun, T. Tian, and H. L. Lin (2016), Optimal linear estimators for systems with finite-step correlated noises and packet dropout compensations. {\it IEEE Trans. Signal Process} vol. 64, no. 21, (5672--5681).
\bibitem[Ahmadi {\em et al.}, 2014]{R14}
A. A. Ahmadi, F. Salmasi, M. Noori Manzar, and T. Najafabadi (2014), Speed sensorless and sensor-fault tolerant optimal PI regulator for networked dc motor system with unknown time-delay and packet dropout. {\it IEEE Trans. Ind. Electron.} vol. 61, no. 2, (708--717).
\bibitem[Nahi, 1969]{R15}
N. E. Nahi (1969), Optimal Recursive Estimation with Uncertain Observation. {\it IEEE Trans. Inf. Theory} vol. 15, no. 4, (457--462).
\bibitem[Sinopoli {\em et al.}, 2004]{R16}
B. Sinopoli, L. Schenato, M. Franceschetti, K. Poolla, M. Jordan, and S. Sastry (2004), Kalman filtering with intermittent observations. {\it IEEE Trans. Autom. Control} vol. 49, no. 9, (1453--1464).
\bibitem[Hadidi {\em et al.}, 1979]{R17}
M. T. Hadidi and C. S. Schwartz (1979), Linear Recursive State Estimators under Uncertain Observations. {\it IEEE Trans. Autom. Control} vol. 24, no. 6, (944--948).
\bibitem[Qi {\em et al.}, 2016]{R18}
Q. Qi and H. Zhang (2016), Output feedback control and stabilization for networked control systems with packet losses. {\it IEEE Trans. Cybern.} DOI: 10.1109/TCYB.2016.2568218.
\bibitem[Zhang {\em et al.}, 2007]{R19}
W. A. Zhang and L. Yu (2007), Output feedback stabilization of networked control systems with packet dropouts. {\it IEEE Trans. Autom. Control} vol. 52, no. 9, (1705--1710).
\bibitem[Gupta {\em et al.}, 2007]{R20}
V. Gupta, B. Hassibi, and R. M. Murray (2007), Optimal LQG control across packet-dropping links. {\it Syst. Control Lett.} vol. 56, no. 6, (439--446).
\bibitem[Liang {\em et al.}, 2016]{R21}
X. Liang and H. Zhang (2016), Linear optimal filter for system subject to random delay and packet dropout. {\it Optim. Control Appl. Meth.} OI: 10.1002/oca.2295.
\bibitem[Xiong {\em et al.}, 2007]{R22}
J. Xiong and J. Lam (2007), Stabilization of linear systems over networks with bounded packet loss. {\it Automatica.} vol. 43, no. 1, (80--87).
\bibitem[Sheng {\em et al.}, 2014]{R23}
L. Sheng, W. Zhang and M. Gao (2014), Relationship between Nash equilibrium strategies and $H_\infty/H_2$ control of stochastic Markov jump systems with multiplicative noise. {\it IEEE Trans. Autom. Control} vol. 59, no. 9, (2592--2597).
\bibitem[Kandil {\em et al.}, 1993]{R24}
H. Abou Kandil, G. Freiling, and G. Jank (1993), Necessary conditions for constant soultions of coupled Riccati equations in Nash games. {\it Syst. Control Lett.} vol. 21, (295--306).
\bibitem[Basar {\em et al.}, 1995]{R25}
T. Basar and G. J. Olsder (1995), {\it Dynamic Noncooperative Game Theory} New York: Academic.
\bibitem[Freiling {\em et al.}, 1999]{R26}
G. Freiling, G. Jank, and H. Abou Kandil (1999), Discrete-time Riccati equations in open-loop Nash and Stackelberg games. {\it Eur. J. Control.} vol. 5, no. 1, (56--66).
\bibitem[Simaan {\em et al.}, 1973]{R27}
M. Simaan and J. B. Cruz (1973), On the stackelberg strategy in nonzero-sum games. {\it J. Optim. Theory Appl.} vol. 11, no. 5, (533--555).
\bibitem[Jungers, 2008]{R28}
M. Jungers (2008), On linear-quadratic Stackelberg games with time preference rates. {\it IEEE Trans. Autom. Control} vol. 53, no. 2, (621--625).
\bibitem[Xu {\em et al.}, 2007]{R29}
J. Xu, H. Zhang, and T. Chai (2015), Necessary and sufficient condition for two-player Stackelberg strategy. {\it IEEE Trans. Autom. Control} vol. 60, no. 5, (1356--1361).
\bibitem[Mehraeen {\em et al.}, 2013]{R30}
S. Mehraeen, T. Dierks, S. Jagannathan, and M. L. Crow (2013), Zero-sum two-player game theoretic formulation of affine nonlinear discrete-time systems using neural networks. {\it IEEE Trans. Cybern.} vol. 43, no. 6, (1641--1655).
\bibitem[Zhang {\em et al.}, 2012]{R31}
H. Zhang, H. Wang, and L. Li (2012), Adapted and casual maximum principle and analytical solution to optimal control for stochastic multiplicative-noise systems with multiple input-delays. {\it in Proc. 51th IEEE Conf. Decision Control} Maui, HI, USA, (2122--2127).
\bibitem[Lin {\em et al.}, 2017]{R32}
H. Lin, H. Su, Z. Shu, P. Shi, R. Lu and Z. G. Wu (2017), Optimal estimation and control for lossy network: stability, convergence, and performance. {\it IEEE Trans. Autom. Control} DOI: 10.1109/TAC.2017.2672729.
\bibitem[Imer {\em et al.}, 2006]{R33}
O. C. Imer, S. Y\"{u}ksel, and T. Ba\c{s}ar (2006), Optimal control of LTI systems over unreliable communication links. {\it Automatica} vol. 42, no. 9, (1429--1439).
\bibitem[Huang {\em et al.}, 2008]{R34}
Y. Huang, W. Zhang, and H. Zhang (2008), Infinite horizon linear quadratic optimal control for discrete-time stochastic systems. {\it Asian J. Control} vol. 10, no. 5, (608--615).
\bibitem[Liang {\em et al.}, 2016]{R35}
X. Liang, J. Xu and H. Zhang (2016), Optimal Control and Stabilization for Networked Control Systems with Packet Dropout and Input Delay. {\it IEEE Trans. Circuits Syst. II, Exp. Briefs} DOI: 10.1109/TCSII.2016.2642986.
\bibitem[Pachter {\em et al.}, 2001]{R36}
M. Pachter, J.J. D¡¯Azzo, and A.W. Proud (2001), Tight formation flight control. {\it Journal of Guidance, Control, and Dynamics} vol. 24, no. 2, (246--254).
\bibitem[Stachnik {\em et al.}, 1984]{R37}
R.V. Stachnik, K. Ashlin, and S. Hamilton (1984), Space station-SAMSI: A spacecraft array for michelson spatial interferometry. {\it Bulletin of the American Astronomical Society} vol. 16, no. 3, (818--827).
\bibitem[Ouyang {\em et al.}, 2016]{R38}
Y. Ouyang, S. M. Asghari, and A. Nayyar (2016), Optimal local and remote controllers with unreliable communication. {\it in Proc. 55th IEEE Conf. Decision Control} Las Vegas, NV, USA, DOI: 10.1109/CDC.2016.7799194.
\bibitem[Bouhtouri {\em et al.}, 1999]{R39}
A. El Bouhtouri, D. Hinrichsen, and A. J. Pritchard (1999), $H_\infty$ type control for discrete-time stochastic systems. {\it Int. J. Robust. Nonlin. Control} vol. 9, no. 13, (923--948).

\end{thebibliography}

\appendix
\section{Proof of Lemma 1}
\begin{proof}
Before proceeding the proof of Lemma 1, we will firstly introduce the following definition. Noting that the local controller $u^L_k$ has access to the states $\{x_0,\ldots,x_k\}$ and the signals $\{y_0,\ldots,y_k\}$, we define
\begin{align}
\hat{u}^L_k=E[u^L_k|\mathcal{F}\{Y_k\}],\label{a}
\end{align}
and
\begin{align}
\tilde{u}^L_k=u^L_k-\hat{u}^L_k.\label{b}
\end{align}
Obviously, we have that
\begin{align}
E[\tilde{u}^L_k|\mathcal{F}\{Y_k\}]\hspace{-0.8mm}=\hspace{-0.8mm}0,E[\tilde{u}^L_k|\mathcal{F}_k]\hspace{-0.8mm}=\hspace{-0.8mm}\tilde{u}^L_k,E[\hat{u}^L_k|\mathcal{F}_k]\hspace{-0.8mm}=\hspace{-0.8mm}\hat{u}^L_k.\label{c}
\end{align}
\emph{Since the local controller $u^L_k$ cannot obtain the remote controller $u^R_k$ at the same time $k$,} (\ref{a})-(\ref{c}) are very useful in the following derivation. Next, we shall rewrite the costate equations (\ref{13}) and (\ref{6}).

Taking mathematical expectation on both sides of (\ref{13}) with $\mathcal{F}\{Y_k\}$, it yields that
\begin{align}
\nonumber0&=E\left[(B^L)'\hspace{-0.8mm}\lambda_k|\mathcal{F}_k|\mathcal{F}\{Y_k\}\right]\hspace{-0.8mm}+\hspace{-0.8mm}E\left[R^Lu^L_k|\mathcal{F}\{Y_k\}\right]\\
          &=E\left[(B^L)'\hspace{-0.8mm}\lambda_k|\mathcal{F}\{Y_k\}\right]\hspace{-0.8mm}+R^L\hat{u}^L_k,\label{d}
\end{align}
which implies that $R^L\hat{u}^L_k=-E\left[(B^L)'\hspace{-0.8mm}\lambda_k|\mathcal{F}\{Y_k\}\right]$. Then noting (\ref{b}), (\ref{13}) can be rewritten as
\begin{align}
\nonumber0&=E\left[(B^L)'\hspace{-0.8mm}\lambda_k|\mathcal{F}_k\right]\hspace{-0.8mm}+\hspace{-0.8mm}R^L\tilde{u}^L_k+\hspace{-0.8mm}R^L\hat{u}^L_k,\\
          &=E\left[(B^L)'\hspace{-0.8mm}\lambda_k|\mathcal{F}_k\hspace{-0.8mm}\right]\hspace{-0.8mm}-\hspace{-0.8mm}E\left[(B^L)'\hspace{-0.8mm}\lambda_k|\mathcal{F}\{Y_k\}\hspace{-0.8mm}\right]\hspace{-0.8mm}+\hspace{-0.8mm}R^L\tilde{u}^L_k.\label{e}
\end{align}
Augmented with (\ref{6}) and (\ref{d}), we have that
\begin{align}
0&=E\left[\begin{bmatrix}B^L&B^R\end{bmatrix}'\hspace{-0.8mm}\lambda_k|\mathcal{F}\{Y_k\}\right]\hspace{-0.8mm}+\hspace{-0.8mm}\begin{bmatrix}R^L&0\\0&R^R\end{bmatrix}\begin{bmatrix}\hat{u}^L_k\\u^R_k\end{bmatrix}\hspace{-0.8mm}.\label{f}
\end{align}
Thus, the costate equations (\ref{13}) and (\ref{6}) can be rewritten as (\ref{e}) and (\ref{f}).

By virtue of (\ref{b}), system (\ref{1}) can be written as
\begin{align}
x_{k+1}=Ax_k+\begin{bmatrix}B^L&B^R\end{bmatrix}\begin{bmatrix}\hat{u}^L_k\\u^R_k\end{bmatrix}+B^L\tilde{u}^L_k+\omega_k,\label{h}
\end{align}

Next, we will show by induction that $\lambda_{k-1}$ has the form as (\ref{31}) for all $k=N+1,\ldots,0$.

Firstly, noting (\ref{5}), (\ref{22}) and $Z_{N+1}=P_{N+1}$, $X_{N+1}=P_{N+1}$, it is obviously that (\ref{31}) holds for $k=N+1$.

For $k=N$, by making use of (\ref{h}), (\ref{5}) and (\ref{c}), (\ref{f}) becomes
\begin{align}
\nonumber0&=E\hspace{-0.8mm}\left[\begin{bmatrix}B^L&\hspace{-0.8mm}B^R\end{bmatrix}'P_{N+1}x_{N+1}|\mathcal{F}\{Y_N\}\right]\hspace{-0.8mm}+\hspace{-0.8mm}\begin{bmatrix}R^L&0\\0&R^R\end{bmatrix}\hspace{-0.8mm}\begin{bmatrix}\hat{u}^L_N\\u^R_N\end{bmatrix}\\
\nonumber&=\begin{bmatrix}B^L&B^R\end{bmatrix}'P_{N+1}A\hat{x}_{N|N}+\begin{bmatrix}R^L&0\\0&R^R\end{bmatrix}\begin{bmatrix}\hat{u}^L_N\\u^R_N\end{bmatrix}\\
\nonumber&\quad+\begin{bmatrix}B^L&B^R\end{bmatrix}'P_{N+1}\begin{bmatrix}B^L&B^R\end{bmatrix}\begin{bmatrix}\hat{u}^L_N\\u^R_N\end{bmatrix}.
\end{align}
Hence, the optimal controller $\begin{bmatrix}\hat{u}^L_N\\u^R_N\end{bmatrix}$ is given by
\begin{align}
\begin{bmatrix}\hat{u}^L_N\\u^R_N\end{bmatrix}&=\hspace{-0.8mm}-\Upsilon_N^{-1}\begin{bmatrix}B^L\hspace{-0.8mm}&\hspace{-0.8mm}B^R\end{bmatrix}'\hspace{-0.8mm}P_{N+1}A\hat{x}_{N|N}=\hspace{-0.8mm}-K_N\hat{x}_{N|N}.\label{15}
\end{align}
Using (\ref{h}), (\ref{5}) and (\ref{c}), we have (\ref{e}) as
\begin{align}
\nonumber0&=E\left[(B^L)'P_{N+1}x_{N+1}|\mathcal{F}_k\right]\hspace{-0.8mm}\\
\nonumber &\quad-E\left[(B^L)'\hspace{-0.8mm}P_{N+1}x_{N+1}|\mathcal{F}\{Y_k\}\right]+R^L\tilde{u}^L_N\\
\nonumber &=(B^L)'P_{N+1}(Ax_k+\begin{bmatrix}B^L&B^R\end{bmatrix}\begin{bmatrix}\hat{u}^L_k\\u^R_k\end{bmatrix}+B^L\tilde{u}^L_k)\\
\nonumber &\quad-(B^L)'P_{N+1}(A\hat{x}_{k|k}+\begin{bmatrix}B^L&B^R\end{bmatrix}\begin{bmatrix}\hat{u}^L_k\\u^R_k\end{bmatrix})+R^L\tilde{u}^L_N\\
\nonumber &=(B^L)'P_{N+1}A\tilde{x}_{N}+(B^L)'P_{N+1}B^L\tilde{u}^L_N+R^L\tilde{u}^L_N.
\end{align}
Thus, the optimal controller $\tilde{u}^L_N$ is derived as
\begin{align}
\nonumber \tilde{u}^L_N&=-\left[(B^L)'P_{N+1}B^L+R^L\right]^{-1}(B^L)'P_{N+1}A\tilde{x}_{N}\\
               &=-\Lambda_N^{-1}M_N\tilde{x}_N.\label{20}
\end{align}
By applying (\ref{h}), (\ref{5}), (\ref{c}), (\ref{15}) and (\ref{20}), it follows from (\ref{12}) that
\begin{align}
\nonumber\lambda_{N-1}&=E\left[A'\lambda_N|\mathcal{F}_N\right]+Qx_{N}\\
\nonumber             &=E\left[A'P_{N+1}x_{N+1}|\mathcal{F}_N\right]+Qx_{N}\\
\nonumber&=A'P_{N+1}Ax_{N}+A'P_{N+1}\begin{bmatrix}B^L\hspace{-1mm}&\hspace{-1mm}B^R\end{bmatrix}\hspace{-0.8mm}\begin{bmatrix}\hat{u}^L_N\\u^R_N\end{bmatrix}\\
\nonumber&\quad+A'P_{N+1}B^L\tilde{u}^L_N+Q\hat{x}_{N|N}+Q\tilde{x}_{N}\\
\nonumber&=\left(A'P_{N+1}A\hspace{-0.8mm}-\hspace{-0.8mm}A'P_{N+1}\hspace{-0.8mm}\begin{bmatrix}B^L&B^R\end{bmatrix}K_N\hspace{-0.8mm}+\hspace{-0.8mm}Q\right)\hat{x}_{N|N}\\
\nonumber&\quad+\left(A'P_{N+1}A-A'P_{N+1}B^L\Lambda_N^{-1}M_N+Q\right)\tilde{x}_{N}.
\end{align}
Noting (\ref{9}) and (\ref{19}), $\lambda_{N-1}$ can be written as
\begin{align}
\nonumber\lambda_{N-1}&=Z_N\hat{x}_{N|N}+X_N\tilde{x}_{N},
\end{align}
which implies that (\ref{31}) holds for $k=N$.

To complete the induction proof, we take any $n$ with $0\leq n\leq N$ and assume that $\lambda_{k-1}$ are as (\ref{31}) for all $k\geq n+1$. We shall show that (\ref{31}) also holds for $k=n$.

Using (\ref{31}), and letting $k=n+1$, $\lambda_n$ can be written as
\begin{align}
 \lambda_n&=Z_{n+1}\hat{x}_{n+1|n+1}+X_{n+1}\tilde{x}_{n+1}.\label{j}
\end{align}
By virtue of (\ref{80}), (\ref{h}) and (\ref{c}), $\hat{x}_{n+1|n+1}$ can be calculated as follows,
\begin{align}
\nonumber\hat{x}_{n+1|n+1}&=\eta_{n+1}x_{n+1}\hspace{-0.8mm}+\hspace{-0.8mm}(1\hspace{-0.8mm}-\hspace{-0.8mm}\eta_{n+1})\hat{x}_{n+1|n}\\
\nonumber &=\eta_{n+1}\left(Ax_n\hspace{-0.8mm}+\hspace{-0.8mm}\begin{bmatrix}B^L&B^R\end{bmatrix}\begin{bmatrix}\hat{u}^L_n\\u^R_n\end{bmatrix}\hspace{-0.8mm}+\hspace{-0.8mm}B^L\tilde{u}^L_n\hspace{-0.8mm}+\hspace{-0.8mm}\omega_n\right)\\
\nonumber&\quad+(1-\eta_{n+1})\left(A\hat{x}_{n|n}+\begin{bmatrix}B^L&B^R\end{bmatrix}\begin{bmatrix}\hat{u}^L_n\\u^R_n\end{bmatrix}\right)\\
\nonumber&=\eta_{n+1}A\tilde{x}_n+A\hat{x}_{n|n}+\begin{bmatrix}B^L&B^R\end{bmatrix}\begin{bmatrix}\hat{u}^L_n\\u^R_n\end{bmatrix}\\
&\quad+\eta_{n+1}B^L\tilde{u}^L_n+\eta_{n+1}\omega_n.\label{k}
\end{align}
With (\ref{22}), (\ref{h}) and (\ref{c}), it can be obtained that
\begin{align}
\nonumber\tilde{x}_{n+1}&=x_{n+1}-\hspace{-0.8mm}\big[\eta_{n+1}x_{n+1}\hspace{-0.8mm}+\hspace{-0.8mm}(1\hspace{-0.8mm}-\hspace{-0.8mm}\eta_{n+1})\hat{x}_{n+1|n}\big]\\
\nonumber               &=Ax_n\hspace{-0.8mm}+\hspace{-0.8mm}\begin{bmatrix}B^L&B^R\end{bmatrix}\begin{bmatrix}\hat{u}^L_n\\u^R_n\end{bmatrix}\hspace{-0.8mm}+\hspace{-0.8mm}B^L\tilde{u}^L_n\hspace{-0.8mm}+\hspace{-0.8mm}\omega_n\\
\nonumber&\quad-\bigg[\eta_{n+1}\left(Ax_n\hspace{-0.8mm}+\hspace{-0.8mm}\begin{bmatrix}B^L&B^R\end{bmatrix}\begin{bmatrix}\hat{u}^L_n\\u^R_n\end{bmatrix}\hspace{-0.8mm}+\hspace{-0.8mm}B^L\tilde{u}^L_n\hspace{-0.8mm}+\hspace{-0.8mm}\omega_n\right)\\
\nonumber&\quad+(1-\eta_{n+1})\left(A\hat{x}_{n|n}+\begin{bmatrix}B^L&B^R\end{bmatrix}\begin{bmatrix}\hat{u}^L_n\\u^R_n\end{bmatrix}\right)\bigg]\\
&=(1-\eta_{n+1})(A\tilde{x}_n+B^L\tilde{u}^L_n+\omega_n).\label{l}
\end{align}
Thus, substituting (\ref{k}) and (\ref{l}) into (\ref{j}), we have that
\begin{align}
\nonumber\lambda_n&=Z_{n+1}\bigg(\eta_{n+1}A\tilde{x}_n+A\hat{x}_{n|n}+\begin{bmatrix}B^L&B^R\end{bmatrix}\begin{bmatrix}\hat{u}^L_n\\u^R_n\end{bmatrix}\\
\nonumber&\qquad\qquad+\eta_{n+1}B^L\tilde{u}^L_n+\eta_{n+1}\omega_n\bigg)\\
         &\quad+X_{n+1}\big[(1\hspace{-0.8mm}-\hspace{-0.8mm}\eta_{n+1})(A\tilde{x}_n+B^L\tilde{u}^L_n+\hspace{-0.8mm}\omega_n)\big].\label{32}
\end{align}
Plugging (\ref{32}) into (\ref{f}), and using (\ref{c}), it yields that
\begin{align}
\nonumber0&=E\left[\begin{bmatrix}B^L&B^R\end{bmatrix}'\hspace{-0.8mm}\lambda_n|\mathcal{F}\{Z_n\}\right]\hspace{-0.8mm}+\hspace{-0.8mm}\begin{bmatrix}R^L&0\\0&R^R\end{bmatrix}\hspace{-0.8mm}\begin{bmatrix}\hat{u}^L_n\\u^R_n\end{bmatrix}\\
\nonumber&=\begin{bmatrix}B^L&\hspace{-0.8mm}B^R\end{bmatrix}'\hspace{-0.8mm}Z_{n+1}A\hat{x}_{n|n}+\hspace{-0.8mm}\begin{bmatrix}R^L\hspace{-0.8mm}&\hspace{-0.8mm}0\\0\hspace{-0.8mm}&\hspace{-0.8mm}R^R\end{bmatrix}\hspace{-0.8mm}\begin{bmatrix}\hat{u}^L_n\\u^R_n\end{bmatrix}\\
\nonumber&\quad+\begin{bmatrix}B^L&\hspace{-0.8mm}B^R\end{bmatrix}'\hspace{-0.8mm}Z_{n+1}\begin{bmatrix}B^L\hspace{-0.8mm}&\hspace{-0.8mm}B^R\end{bmatrix}\begin{bmatrix}\hat{u}^L_n\\u^R_n\end{bmatrix}.
\end{align}
The optimal controller $\begin{bmatrix}\hat{u}^L_n\\u^R_n\end{bmatrix}$ is derived as
\begin{align}
\begin{bmatrix}\hat{u}^L_n\\u^R_n\end{bmatrix}&\hspace{-0.8mm}=\hspace{-0.8mm}-\hspace{-0.8mm}\Upsilon_n^{-1}\begin{bmatrix}B^L&\hspace{-0.8mm}B^R\end{bmatrix}'\hspace{-0.8mm}Z_{n+1}A\hat{x}_{n|n}\hspace{-0.8mm}=\hspace{-0.8mm}-K_n\hat{x}_{n|n}.\label{17}
\end{align}
On the other hand, substituting (\ref{32}) into (\ref{e}) and using (\ref{c}), we get
\begin{align}
\nonumber0&=E\left[(B^L)'\hspace{-0.8mm}\lambda_n|\mathcal{F}_n\right]\hspace{-0.8mm}-E\left[(B^L)'\hspace{-0.8mm}\lambda_n|\mathcal{F}\{Z_n\}\right]\hspace{-0.8mm}+\hspace{-0.8mm}R^L\tilde{u}^L_n\\
\nonumber&=(1-p)(B^L)'Z_{n+1}A\tilde{x}_n+(1-p)(B^L)'Z_{n+1}B^L\tilde{u}^L_n\\
\nonumber&\quad+p(B^L)'X_{n+1}A\tilde{x}_n\hspace{-0.8mm}+p(B^L)'X_{n+1}B^L\tilde{u}^L_n\hspace{-0.8mm}+R^L\tilde{u}^L_n.
\end{align}
Thus, the optimal controller $\tilde{u}^L_n$ is given by
\begin{align}
\nonumber \tilde{u}^L_n&=-\big[p(B^L)'X_{n+1}B^L\hspace{-0.8mm}+\hspace{-0.8mm}(1\hspace{-0.8mm}-\hspace{-0.8mm}p)(B^L)'Z_{n+1}B^L\hspace{-0.8mm}+\hspace{-0.8mm}R^L\big]^{-1}\\
\nonumber      &\quad\times\big[p(B^L)'X_{n+1}A+(1-p)(B^L)'Z_{n+1}A\big]\tilde{x}_n\\
               &=-\Lambda_n^{-1}M_n\tilde{x}_{n}.\label{23}
\end{align}
Now we show that for $k=n$, $\lambda_{n-1}$ is as the form of (\ref{31}). Using (\ref{12}) and (\ref{32}), it yields that
\begin{align}
\nonumber\lambda_{n-1}&=E\left[A'\lambda_n|\mathcal{F}_n\right]+Qx_{n}\\
\nonumber&=E\bigg\{A'Z_{n+1}\bigg(\eta_{n+1}A\tilde{x}_n\hspace{-0.8mm}+\hspace{-0.8mm}A\hat{x}_{n|n}\hspace{-0.8mm}+\hspace{-0.8mm}\begin{bmatrix}B^L&B^R\end{bmatrix}\begin{bmatrix}\hat{u}^L_n\\u^R_n\end{bmatrix}\\
\nonumber&\quad+\eta_{n+1}B^L\tilde{u}^L_n+\eta_{n+1}\omega_n\bigg)+A'X_{n+1}\big[(1-\eta_{n+1})\\
\nonumber&\quad\times(A\tilde{x}_n+B^L\tilde{u}^L_n+\omega_n)\big]|\mathcal{F}_n\bigg\}+Qx_{n}\\
\nonumber&=A'Z_{n+1}A\hat{x}_{n|n}+\hspace{-0.8mm}A'Z_{n+1}\begin{bmatrix}B^L&B^R\end{bmatrix}\begin{bmatrix}\hat{u}^L_n\\u^R_n\end{bmatrix}\hspace{-0.8mm}\hspace{-0.8mm}+\hspace{-0.8mm}Q\hat{x}_{n|n}\\
\nonumber&\quad+(1-p)A'Z_{n+1}A\tilde{x}_n+(1-p)A'Z_{n+1}B^L\tilde{u}^L_n\\
\nonumber&\quad+pA'X_{n+1}A\tilde{x}_n+pA'X_{n+1}B^L\tilde{u}^L_n+Q\tilde{x}_n.
\end{align}
Substituting (\ref{17}), (\ref{23}), (\ref{9}) and (\ref{19}) into the above equation, we have
\begin{align}
\nonumber\lambda_{n-1}&=\big(A'Z_{n+1}A-A'Z_{n+1}\begin{bmatrix}B^L\hspace{-0.8mm}&\hspace{-0.8mm}B^R\end{bmatrix}K_n+Q\big)\hat{x}_{n|n}\\
\nonumber&\quad+\big\{pA'X_{n+1}A+(1-p)A'Z_{n+1}A\\
\nonumber&\quad-\hspace{-0.8mm}[pA'X_{n+1}B^L\hspace{-0.8mm}+\hspace{-0.8mm}(1\hspace{-0.8mm}-\hspace{-0.8mm}p)A'Z_{n+1}B^L\hspace{-0.8mm}]\Lambda_n^{-1}M_n\hspace{-0.8mm}+\hspace{-0.8mm}Q\hspace{-0.8mm}\big\}\hspace{-0.8mm}\tilde{x}_n\\
\nonumber&=Z_n\hat{x}_{n|n}+X_n\tilde{x}_n.
\end{align}
Thus (\ref{31}) holds for $k=n$. This completes the proof.
\end{proof}
\section{Proof of Theorem 1}
\begin{proof}``\emph{Necessity}'': Suppose \emph{Problem 1} has the unique solution. We will show by induction that $\Upsilon_k>0$ and $\Lambda_k>0$ for $k=N,\ldots,0$.

Noting (\ref{b}), the cost function (\ref{2}) can be written as
\begin{align}
\nonumber J_N=&E\bigg\{\sum_{k=0}^N\bigg[{x_k}'Qx_k+\begin{bmatrix}\hat{u}^L_k\\u^R_k\end{bmatrix}'\begin{bmatrix}R^L&0\\0&R^R\end{bmatrix}\begin{bmatrix}\hat{u}^L_k\\u^R_k\end{bmatrix}\\
              &\quad\quad+(\tilde{u}^L_{k})'R^L\tilde{u}^L_k\bigg]+{x_{N+1}}'P_{N+1}x_{N+1}\bigg\}.\label{10}
\end{align}
Define
\begin{align}
\nonumber J(k)=&E\bigg\{\sum_{i=k}^N\bigg[{x_i}'Qx_i+\begin{bmatrix}\hat{u}^L_i\\u^R_i\end{bmatrix}'\begin{bmatrix}R^L&0\\0&R^R\end{bmatrix}\begin{bmatrix}\hat{u}^L_i\\u^R_i\end{bmatrix}\\
\nonumber&\quad\quad+(\tilde{u}^L_{i})'R^L\tilde{u}^L_{i}\bigg]+{x_{N+1}}'P_{N+1}x_{N+1}\bigg\}
\end{align}
for $k=0,\ldots,N$.

Firstly, for $k=N$, note that
\begin{align}
\nonumber J(N)&=E\bigg[{x_N}'Qx_N+\begin{bmatrix}\hat{u}^L_N\\u^R_N\end{bmatrix}'\begin{bmatrix}R^L&0\\0&R^R\end{bmatrix}\begin{bmatrix}\hat{u}^L_N\\u^R_N\end{bmatrix}\\
\nonumber&\qquad+(\tilde{u}^L_{N})'R^L\tilde{u}^L_{N}+{x_{N+1}}'P_{N+1}x_{N+1}\bigg].
\end{align}
Using (\ref{h}) and setting $x_N=0$, the above equation can be written as
\begin{align}
\nonumber J(N)&=\begin{bmatrix}\hat{u}^L_N\\u^R_N\end{bmatrix}'\begin{bmatrix}R^L&0\\0&R^R\end{bmatrix}\begin{bmatrix}\hat{u}^L_N\\u^R_N\end{bmatrix}+(\tilde{u}^L_{N})'R^L\tilde{u}^L_{N}\\
\nonumber&+\begin{bmatrix}\hat{u}^L_N\\u^R_N\end{bmatrix}'\begin{bmatrix}B^L&B^R\end{bmatrix}'P_{N+1}\begin{bmatrix}B^L&B^R\end{bmatrix}\begin{bmatrix}\hat{u}^L_N\\u^R_N\end{bmatrix}\\
\nonumber&+(\tilde{u}^L_{N})'(B^L)'P_{N+1}B^L\tilde{u}^L_{N}+Tr(P_{N+1}Q_{\omega_N}).
\end{align}
By applying (\ref{28}) and (\ref{29}), the above equation becomes
\begin{align}
\nonumber J(N)&=\begin{bmatrix}\hat{u}^L_N\\u^R_N\end{bmatrix}'\hspace{-0.8mm}\Upsilon_N\hspace{-0.8mm}\begin{bmatrix}\hat{u}^L_N\\u^R_N\end{bmatrix}\hspace{-0.8mm}+\hspace{-0.8mm}(\tilde{u}^L_{N})'\Lambda_N\tilde{u}^L_{N}\hspace{-0.8mm}+\hspace{-0.8mm}Tr(P_{N+1}Q_{\omega_N})\\
\nonumber&=\begin{bmatrix}\hat{u}^L_N\\u^R_N\\\tilde{u}^L_N\end{bmatrix}'\hspace{-0.8mm}\begin{bmatrix}\Upsilon_N&0&0\\0&\Upsilon_N&0\\0&0&\Lambda_N\end{bmatrix}\hspace{-0.8mm}\begin{bmatrix}\hat{u}^L_N\\u^R_N\\\tilde{u}^L_N\end{bmatrix}\hspace{-0.8mm}+\hspace{-0.8mm}Tr(P_{N+1}Q_{\omega_N}).
\end{align}
The uniqueness of the optimal $\hat{u}^L_N$, $u^R_N$ and $\tilde{u}^L_N$ implies that the quadratic term $\begin{bmatrix}\Upsilon_N&0&0\\0&\Upsilon_N&0\\0&0&\Lambda_N\end{bmatrix}$ is positive for any nonzero $\begin{bmatrix}\hat{u}^L_N\\u^R_N\\\tilde{u}^L_N\end{bmatrix}$. Thus, we have that $\Upsilon_N>0$ and $\Lambda_N>0$.

Next, let any $n$ with $0\leq n\leq N$, and assume that $\Upsilon_k>0$ and $\Lambda_k>0$ for all $k\geq n+1$. We shall show that $\Upsilon_k>0$ and $\Lambda_k>0$ for $k=n$.

Using (\ref{1}), (\ref{12}), (\ref{13}) and (\ref{6}), for $k\geq n+1$, we get
\begin{align}
\nonumber&E[x_k'\lambda_{k-1}-x_{k+1}'\lambda_k]\\
\nonumber&=E\bigg[x_k'E\left[A'\lambda_k|\mathcal{F}_k\right]+x_k'Qx_k-x_k'A'\lambda_k\\
\nonumber&\qquad-(u^R_k)'(B^R)'\lambda_k-(u^L_k)'(B^L)'\lambda_k-\omega_k'\lambda_k\bigg]\\
\nonumber&=E[x_k'Qx_k]-E\left\{(u^R_k)'E\left[(B^R)'\lambda_k|\mathcal{F}\{Y_k\}\right]\right\}\\
\nonumber&\quad-E\left\{({u}^L_k)'E\left[(B^L)'\lambda_k|\mathcal{F}_k\right]\right\}-E[\omega_k'\lambda_k]\\
\nonumber&=E\left[x_k'Qx_k\hspace{-0.8mm}+\hspace{-0.8mm}(u^R_k)'R^Ru^R_k\hspace{-0.8mm}+\hspace{-0.8mm}({u}^L_k)'R^L{u}^L_k\right]-E[\omega_k'\lambda_k].
\end{align}
Taking summation from $k=n+1$ to $k=N$ on both sides of the above equation, it yields that
\begin{align}
\nonumber&E[x_{n+1}'\lambda_{n}-x_{N+1}'\lambda_{N}]\\
\nonumber&=E[x_{n+1}'\lambda_{n}-x_{N+1}'P_{N+1}x_{N+1}]\\
\nonumber&=\hspace{-1mm}\sum_{k=n+1}^N\hspace{-0.8mm}\{\hspace{-0.8mm}E\hspace{-0.8mm}\left[\hspace{-0.8mm}x_k'Qx_k\hspace{-0.8mm}+\hspace{-0.8mm}(u^R_k)'R^Ru^R_k\hspace{-0.8mm}+\hspace{-0.8mm}({u}^L_k)'R^L{u}^L_k\hspace{-0.8mm}\right]\hspace{-0.8mm}-\hspace{-0.8mm}E(\omega_k'\lambda_k)\hspace{-1mm}\}.
\end{align}
Obviously, we obtain
\begin{align}
\nonumber&E[x_{n+1}'\lambda_{n}]=\hspace{-0.8mm}\sum_{k=n+1}^N\hspace{-0.8mm}E\bigg[x_k'Qx_k\hspace{-0.8mm}+\hspace{-0.8mm}(u^R_k)'R^Ru^R_k\hspace{-0.8mm}+\hspace{-0.8mm}({u}^L_k)'R^L{u}^L_k\\
\nonumber&\qquad\qquad\qquad+x_{N+1}'P_{N+1}x_{N+1}\bigg]-\sum_{k=n+1}^NE(\omega_k'\lambda_k).
\end{align}
Applying (\ref{b}), the above equation becomes
\begin{align}
\nonumber&E[x_{n+1}'\lambda_{n}]=\hspace{-0.8mm}\sum_{k=n+1}^N\hspace{-0.8mm}E\bigg[x_k'Qx_k\hspace{-0.8mm}+\hspace{-0.8mm}\begin{bmatrix}\hat{u}^L_k\\u^R_k\end{bmatrix}'\begin{bmatrix}R^L&0\\0&R^R\end{bmatrix}\begin{bmatrix}\hat{u}^L_k\\u^R_k\end{bmatrix}\hspace{-0.8mm}\\
\nonumber&+(\tilde{u}^L_{k})'R^L\tilde{u}^L_{k}+x_{N+1}'P_{N+1}x_{N+1}\bigg]-\sum_{k=n+1}^NE(\omega_k'\lambda_k).
\end{align}
Thus, we have
\begin{align}
\nonumber J(n)&=E\bigg[x_n'Qx_n\hspace{-0.8mm}+\hspace{-0.8mm}\begin{bmatrix}\hat{u}^L_n\\u^R_n\end{bmatrix}'\begin{bmatrix}R^L&0\\0&R^R\end{bmatrix}\hspace{-0.8mm}\begin{bmatrix}\hat{u}^L_n\\u^R_n\end{bmatrix}\hspace{-0.8mm}+\hspace{-0.8mm}(\tilde{u}^L_{n})'R^L\tilde{u}^L_{n}\bigg]\\
\nonumber&\quad+\hspace{-0.8mm}\sum_{k=n+1}^N\hspace{-0.8mm}E\bigg[x_k'Qx_k\hspace{-0.8mm}+\hspace{-0.8mm}\begin{bmatrix}\hat{u}^L_k\\u^R_k\end{bmatrix}'\hspace{-0.8mm}\begin{bmatrix}R^L\hspace{-0.8mm}&\hspace{-0.8mm}0\\0\hspace{-0.8mm}&\hspace{-0.8mm}R^R\end{bmatrix}\begin{bmatrix}\hat{u}^L_k\\u^R_k\end{bmatrix}\hspace{-0.8mm}\\
\nonumber&\quad+(\tilde{u}^L_{k})'R^L\tilde{u}^L_{k}+x_{N+1}'P_{N+1}x_{N+1}\bigg]\\
\nonumber&=E\bigg[x_n'Qx_n\hspace{-0.8mm}+\hspace{-0.8mm}\begin{bmatrix}\hat{u}^L_n\\u^R_n\end{bmatrix}'\begin{bmatrix}R^L&0\\0&R^R\end{bmatrix}\hspace{-0.8mm}\begin{bmatrix}\hat{u}^L_n\\u^R_n\end{bmatrix}\hspace{-0.8mm}+\hspace{-0.8mm}(\tilde{u}^L_{n})'R^L\tilde{u}^L_{n}\bigg]\\
\nonumber&\quad+E[x_{n+1}'\lambda_{n}]+\sum_{k=n+1}^NE(\omega_k'\lambda_k)
\end{align}
Since $\Upsilon_k>0$ and $\Lambda_k>0$ for $k\geq n+1$, noting Lemma 1, we have $\lambda_n=Z_{n+1}\hat{x}_{n+1|n+1}+X_{n+1}\tilde{x}_{n+1}$. Setting $x_n=0$ (thus $\hat{x}_{n|n}=0$ and $\tilde{x}_n=0$) and using (\ref{h}) and (\ref{32}), the above equation becomes
\begin{align}
\nonumber J(n)&=E\bigg[\begin{bmatrix}\hat{u}^L_n\\u^R_n\end{bmatrix}'\begin{bmatrix}R^L&0\\0&R^R\end{bmatrix}\begin{bmatrix}\hat{u}^L_n\\u^R_n\end{bmatrix}\hspace{-0.8mm}+\hspace{-0.8mm}(\tilde{u}^L_n)'R^L\tilde{u}^L_n\bigg]\\
\nonumber&\quad+E\bigg[\begin{bmatrix}\hat{u}^L_n\\u^R_n\end{bmatrix}'\begin{bmatrix}B^L\hspace{-0.8mm}&\hspace{-0.8mm}B^R\end{bmatrix}'\lambda_n\hspace{-0.8mm}+\hspace{-0.8mm}(\tilde{u}^L_n)'(B^L)'\lambda_n\bigg]\\
\nonumber&\quad+\sum_{k=n+1}^NE(\omega_k'\lambda_k)\\
\nonumber&=E\bigg[\begin{bmatrix}\hat{u}^L_n\\u^R_n\end{bmatrix}'\begin{bmatrix}R^L&0\\0&R^R\end{bmatrix}\begin{bmatrix}\hat{u}^L_n\\u^R_n\end{bmatrix}\hspace{-0.8mm}+\hspace{-0.8mm}(\tilde{u}^L_n)'R^L\tilde{u}^L_n\bigg]\\
\nonumber&\quad+E\bigg[\begin{bmatrix}\hat{u}^L_n\\u^R_n\end{bmatrix}'\begin{bmatrix}B^L\hspace{-0.8mm}&\hspace{-0.8mm}B^R\end{bmatrix}'Z_{n+1}\begin{bmatrix}B^L&B^R\end{bmatrix}\begin{bmatrix}\hat{u}^L_n\\u^R_n\end{bmatrix}\\
\nonumber&\quad+(1-p)(\tilde{u}^L_n)'(B^L)'Z_{n+1}B^L\tilde{u}^L_n\\
\nonumber&\quad+p(\tilde{u}^L_n)'(B^L)'X_{n+1}B^L\tilde{u}^L_n\bigg]+\sum_{k=n+1}^NE(\omega_k'\lambda_k).
\end{align}
Making use of (\ref{28}) and (\ref{29}), the above equation can be written as
\begin{align}
\nonumber J(n)&=\begin{bmatrix}\hat{u}^L_n\\u^R_n\end{bmatrix}'\Upsilon_n\begin{bmatrix}\hat{u}^L_n\\u^R_n\end{bmatrix}+(\tilde{u}^L_n)'\Lambda_n\tilde{u}^L_n+\sum_{k=n+1}^NE(\omega_k'\lambda_k)\\
\nonumber&=\begin{bmatrix}\hat{u}^L_n\\u^R_n\\\tilde{u}^L_n\end{bmatrix}'\begin{bmatrix}\Upsilon_n&0&0\\0&\Upsilon_n&0\\0&0&\Lambda_n\end{bmatrix}\begin{bmatrix}\hat{u}^L_n\\u^R_n\\\tilde{u}^L_n\end{bmatrix}+\sum_{k=n+1}^NE(\omega_k'\lambda_k).
\end{align}
The uniqueness of the optimal $\hat{u}^L_n$, $u^R_n$ and $\tilde{u}^L_n$ implies that the quadratic term $\begin{bmatrix}\Upsilon_n&0&0\\0&\Upsilon_n&0\\0&0&\Lambda_n\end{bmatrix}$ is positive for any nonzero $\begin{bmatrix}\hat{u}^L_n\\u^R_n\\\tilde{u}^L_n\end{bmatrix}$. Therefore, it follows that $\Upsilon_n>0$ and $\Lambda_n>0$. The proof of the necessity is completed.

``\emph{Sufficiency}'': Suppose that $\Upsilon_k>0$ and $\Lambda_k>0$ for $0\leq k\leq N$. The uniqueness of the solution to \emph{Problem 1} is to be shown.

Define
\begin{align}
V_N(k,x_k)=E\left[x_k'Z_k\hat{x}_{k|k}+x_k'X_{k}\tilde{x}_k\right].
\end{align}
Using (\ref{h}), (\ref{9})-(\ref{30}) and (\ref{c}), we have
\begin{align}
\nonumber&V_N(k,x_k)-V_N(k+1,x_{k+1})\\
\nonumber&=E[x_k'Z_kx_k-\tilde{x}_k'Z_k\tilde{x}_k+\tilde{x}_k'X_k\tilde{x}_k]\\
\nonumber&\quad-E\bigg\{\bigg(Ax_k+\begin{bmatrix}B^L&B^R\end{bmatrix}\begin{bmatrix}\hat{u}^L_k\\u^R_k\end{bmatrix}+B^L\tilde{u}^L_k+\omega_k\bigg)'Z_{k+1}\\
\nonumber&\quad\times\big[\eta_{k+1}x_{k+1}\hspace{-0.8mm}+\hspace{-0.8mm}(1\hspace{-0.8mm}-\hspace{-0.8mm}\eta_{k+1})\hat{x}_{k+1|k}\big]\hspace{-0.8mm}+\hspace{-0.8mm}\bigg(\hspace{-0.8mm}Ax_k\hspace{-0.8mm}+\hspace{-0.8mm}\begin{bmatrix}B^L\hspace{-0.8mm}&\hspace{-0.8mm}B^R\end{bmatrix}\\
\nonumber&\quad\times\begin{bmatrix}\hat{u}^L_k\\u^R_k\end{bmatrix}\hspace{-0.8mm}+\hspace{-0.8mm}B^L\tilde{u}^L_k\hspace{-0.8mm}+\hspace{-0.8mm}\omega_k\hspace{-0.8mm}\bigg)'X_{k+1}\big[x_{k+1}\hspace{-0.8mm}-\hspace{-0.8mm}\big(\eta_{k+1}x_{k+1}\hspace{-0.8mm}\\
\nonumber&\quad+\hspace{-0.8mm}(1\hspace{-0.8mm}-\hspace{-0.8mm}\eta_{k+1})\hat{x}_{k+1|k}\big)\big]\bigg\}\\
\nonumber&=E\bigg\{x_k'(Z_k-A'Z_{k+1}A+K_k'\Upsilon_kK_k)x_k\\
\nonumber&\quad-\begin{bmatrix}\hat{u}^L_k\\u^R_k\end{bmatrix}'\bigg(\Upsilon_k\hspace{-0.8mm}-\hspace{-0.8mm}\begin{bmatrix}R^L\hspace{-0.8mm}&\hspace{-0.8mm}0\\0\hspace{-0.8mm}&\hspace{-0.8mm}R^R\end{bmatrix}\bigg)\hspace{-0.8mm}\begin{bmatrix}\hat{u}^L_k\\u^R_k\end{bmatrix}\hspace{-0.8mm}-\hspace{-0.8mm}2\begin{bmatrix}\hat{u}^L_k\\u^R_k\end{bmatrix}'\hspace{-0.8mm}\Upsilon_kK_k\hat{x}_{k|k}\\
\nonumber&\quad-\hat{x}_{k|k}'K_k'\Upsilon_kK_k\hat{x}_{k|k}\hspace{-0.8mm}-\hspace{-0.8mm}(\hspace{-0.8mm}\tilde{u}^L_k\hspace{-0.8mm})'(\Lambda_k\hspace{-0.8mm}-\hspace{-0.8mm}R^L)\tilde{u}^L_k\hspace{-0.8mm}-\hspace{-0.8mm}2(\tilde{u}^L_k)'M_k\tilde{x}_k\\
\nonumber&\quad-\tilde{x}_k'(pA'X_{k+1}A\hspace{-0.8mm}-\hspace{-0.8mm}pA'Z_{k+1}A\hspace{-0.8mm}-\hspace{-0.8mm}X_k\hspace{-0.8mm}+\hspace{-0.8mm}Z_k\hspace{-0.8mm}+\hspace{-0.8mm}K_k'\Upsilon_kK_k)\\
\nonumber&\quad
\times\tilde{x}_k\bigg\}-pTr(Q_{\omega_k}X_{k+1})-(1-p)Tr(Q_{\omega_k}Z_{k+1})\\
\nonumber&=E\bigg\{x_k'Qx_k+\begin{bmatrix}\hat{u}^L_k\\u^R_k\end{bmatrix}'\begin{bmatrix}R^L\hspace{-0.8mm}&\hspace{-0.8mm}0\\0\hspace{-0.8mm}&\hspace{-0.8mm}R^R\end{bmatrix}\begin{bmatrix}\hat{u}^L_k\\u^R_k\end{bmatrix}+(\tilde{u}^L_k)'R^L\tilde{u}^L_k\\
\nonumber&\quad-\left(\begin{bmatrix}\hat{u}^L_k\\u^R_k\end{bmatrix}+K_k\hat{x}_{k|k}\right)'\Upsilon_k\left(\begin{bmatrix}\hat{u}^L_k\\u^R_k\end{bmatrix}+K_k\hat{x}_{k|k}\right)\\
\nonumber&\quad-(\tilde{u}^L_k+\Lambda_k^{-1}M_k\tilde{x}_k)'\Lambda_k(\tilde{u}^L_k+\Lambda_k^{-1}M_k\tilde{x}_k)\bigg\}\\
         &\quad-pTr(Q_{\omega_k}X_{k+1})-(1-p)Tr(Q_{\omega_k}Z_{k+1}).\label{33}
\end{align}
Taking summation from $k=0$ to $k=N$ on both sides of (\ref{33}), the cost function (\ref{10}) can be written as
\begin{align}
\nonumber J_N&=E\big[x_0'Z_0\hat{x}_{0|0}\hspace{-0.8mm}+\hspace{-0.8mm}x_0'X_0\tilde{x}_0\big]\hspace{-0.8mm}+\hspace{-0.8mm}E\sum^N_{k=0}\bigg\{\hspace{-0.8mm}\left(\hspace{-0.8mm}\begin{bmatrix}\hat{u}^L_k\\u^R_k\end{bmatrix}\hspace{-0.8mm}+\hspace{-0.8mm}K_k\hat{x}_{k|k}\hspace{-0.8mm}\right)'\\
\nonumber&\quad\times\Upsilon_k\left(\begin{bmatrix}\hat{u}^L_k\\u^R_k\end{bmatrix}+K_k\hat{x}_{k|k}\right)+(\tilde{u}^L_k+\Lambda_k^{-1}M_k\tilde{x}_k)'\\
\nonumber&\quad\times\Lambda_k(\tilde{u}^L_k+\Lambda_k^{-1}M_k\tilde{x}_k)\bigg\}+\sum^N_{k=0}[pTr(Q_{\omega_k}X_{k+1})\\
\nonumber&\quad+(1-p)Tr(Q_{\omega_k}Z_{k+1})].
\end{align}
Note that $\Upsilon_k>0$ and $\Lambda_k>0$ for $k=0,\ldots,N$. Thus, with (\ref{b}), the unique optimal controllers $u^R_k$ and ${u}^L_k$ exist and are given by (\ref{8}) and (\ref{14}). Accordingly, we have the optimal cost as
\begin{align}
\nonumber J^*_N&=E\big[x_0'Z_0\hat{x}_{0|0}+x_0'X_0\tilde{x}_0\big]+\sum^N_{k=0}[pTr(Q_{\omega_k}X_{k+1})\\
&\quad+(1-p)Tr(Q_{\omega_k}Z_{k+1})],\label{34}
\end{align}
which is exactly the value of (\ref{25}). This completes the sufficiency proof.
\end{proof}
\section{Proof of Theorem 2}
\begin{proof}
``Necessity'': Under Assumptions 1 and 2, suppose the system (\ref{16}) is stabilizable in the mean-square sense. We will show that there exist the unique solutions $Z$ and $X$ to the two Riccati equations (\ref{24}) and (\ref{26}), such that $Z>0$, $\Psi>0$ and $E[{x}_{0}'Z\hat{x}_{0|0}+{x}_0'X\tilde{x}_0]\geq0$ for any initial value $x_0$.

With (\ref{b}), the cost function (\ref{21}) can be written as
\begin{align} J\hspace{-0.8mm}&=\hspace{-0.8mm}E\sum_{k=0}^\infty\bigg[{x_k}'Qx_k\hspace{-0.8mm}+\hspace{-0.8mm}\begin{bmatrix}\hat{u}^L_k\\u^R_k\end{bmatrix}'\begin{bmatrix}R^L\hspace{-0.8mm}&\hspace{-0.8mm}0\\0\hspace{-0.8mm}&\hspace{-0.8mm}R^R\end{bmatrix}\hspace{-0.8mm}\begin{bmatrix}\hat{u}^L_k\\u^R_k\end{bmatrix}\hspace{-0.8mm}+\hspace{-0.8mm}(\tilde{u}^L_k)'R^L\tilde{u}^L_k\bigg]\label{g}
\end{align}
The transformation of (\ref{g}) is convenient for the following proof.

Firstly, we will prove that $Z_k(N)$ and $X_k(N)$ are convergent. Observing (\ref{9}) and (\ref{19}), it can be obtained that $Z_k(N)$ and $X_k(N)$ are uncorrelated with the initial value $x_0$. Since the additive noise is not considered in this section, let $\bar{x}_0=0$ and with (\ref{83}) and (\ref{84}), the optimal cost (\ref{34}) can be written as
\begin{align}
\nonumber J^*_N&=E\big[x_0'Z_0(N)\hat{x}_{0|0}\hspace{-0.8mm}+\hspace{-0.8mm}x_0'X_0(N)\tilde{x}_0\big]\\
\nonumber      &=E\left\{x_0'\left[(1-p)Z_0(N)\hspace{-0.8mm}+\hspace{-0.8mm}pX_0(N)\right]x_0\right\}\\
               &=E\left[x_0'\Psi_0(N)x_0\right]. \label{48}
\end{align}
Thus, we have
\begin{align}
E\left[x_0'\Psi_0(N)x_0\right]=J^*_N\leq J^*_{N+1}=E\left[x_0'\Psi_0(N\hspace{-0.8mm}+\hspace{-0.8mm}1)x_0\right].\label{46}
\end{align}
Due to the arbitrariness of $x_0$, it can be obtained that $\Psi_0(N)$ increases with respect to $N$.

Now the boundedness of $\Psi_0(N)$ is to be shown. Since system (\ref{16}) is stabilizable in the mean-square sense, there exist $u^R_k=-K_1\hat{x}_{k|k}$, and $\tilde{u}^L_k=-K_2\hat{x}_{k|k}-L\tilde{x}_{k}$ with constant matrices $K_1$, $K_2$ and $L$, such that the closed-loop system (\ref{16}) satisfies
\begin{align}
\lim_{k\to \infty}E(x_k'x_k)=0.\label{44}
\end{align}
With (\ref{22}), we have
\begin{align}
\nonumber \lim_{k\to\infty}E[x_k'x_k]&=\lim_{k\to\infty}E[(\hat{x}_{k|k}+\tilde{x}_{k|k})'(\hat{x}_{k|k}+\tilde{x}_{k|k})]\\
                                     &=\lim_{k\to\infty}\left(E[\hat{x}_{k|k}'\hat{x}_{k|k}]+E[\tilde{x}_k'\tilde{x}_k]\right).\label{45}
\end{align}
Combining (\ref{44}) and (\ref{45}), it yields that
\begin{align}
\lim_{k\to\infty}E[\hat{x}_{k|k}'\hat{x}_{k|k}]=0,\quad \lim_{k\to\infty}E[\tilde{x}_k'\tilde{x}_k]=0.\label{50}
\end{align}
By \cite{R39}, there exist constants $c_1>0$, $c_2>0$ and $c_3>0$ satisfying
\begin{align}
\nonumber \sum^\infty_{k=0}E[x_k'x_k]&\leq c_1E[x_0'x_0],\\
\nonumber \sum^\infty_{k=0}E[\hat{x}_{k|k}'\hat{x}_{k|k}]&\leq c_2E[\hat{x}_{0|0}'\hat{x}_{0|0}],\\
\nonumber \sum^\infty_{k=0}E[\tilde{x}_k'\tilde{x}_k]&\leq c_3E[\tilde{x}_0'\tilde{x}_0].
\end{align}
Select a constant $c_4$, such that $Q\leq c_4I$, $K_1R^RK_1\leq c_4I$, $K_2'R^LK_2\leq c_4I$ and $L'R^LL\leq c_4I$. Then, we have
\begin{align}
\nonumber J&=E\sum_{k=0}^\infty\bigg[{x_k}'Qx_k\hspace{-0.8mm}+\hspace{-0.8mm}(u^R_k)'R^Ru^R_k+\hspace{-0.8mm}({u}^L_k)'R^L{u}^L_k\bigg]\\
\nonumber  &=E\sum_{k=0}^\infty[{x_k}'Qx_k]+E\sum_{k=0}^\infty\bigg[\hat{x}_{k|k}'K_1'R^RK_1\hat{x}_{k|k}\bigg]\\
\nonumber  &\quad++E\sum_{k=0}^\infty\bigg[\hat{x}_{k|k}'K_2'R^RK_2\hat{x}_{k|k}\bigg]+E\sum_{k=0}^\infty[\tilde{x}_k'L'R_L\tilde{x}_k]\\
\nonumber  &\leq c_4\left\{E\sum_{k=0}^\infty[{x_k}'x_k]\hspace{-0.8mm}+\hspace{-0.8mm}E\sum_{k=0}^\infty\left[\hat{x}_{k|k}'\hat{x}_{k|k}\right]\hspace{-0.8mm}+\hspace{-0.8mm}E\sum_{k=0}^\infty[\tilde{x}_k'\tilde{x}_k]\right\}\\
\nonumber  &\leq c_4\left\{c_1E[x_0'x_0]+c_2E[\hat{x}_{0|0}'\hat{x}_{0|0}]+c_3E[\tilde{x}_0'\tilde{x}_0]\right\}
\end{align}
From (\ref{48}), for any $N>0$, we have
\begin{align}
\nonumber &E\left[x_0'\Psi_0(N)x_0\right]=J^*_N\leq J\\
\nonumber&\leq c_4\left\{c_1E[x_0'x_0]+c_2E[\hat{x}_{0|0}'\hat{x}_{0|0}]+c_3E[\tilde{x}_0'\tilde{x}_0]\right\}
\end{align}
which implies that $\Psi_0(N)$ is bounded. Recall that $\Psi_0(N)$ is monotonically increasing. Thus, $\Psi_0(N)$ is convergent, i.e.,
\begin{align}
\nonumber\lim_{N\to\infty}\Psi_0(N)=\Psi.
\end{align}
Note that the variables given in (\ref{9})-(\ref{30}) are time invariant for $N$ due to the choice that $P_{N+1}=0$, i.e,
\begin{align}
\nonumber Z_k(N)&=Z_{k-s}(N-s), X_k(N)=X_{k-s}(N-s),\\
\nonumber K_k(N)&=K_{k-s}(N-s), \Upsilon_k=\Upsilon_{k-s}(N-s),\\
\nonumber\Psi_k(N)&=\Psi_{k-s}(N-s), \Lambda_k(N)=\Lambda_{k-s}(N-s),\\
\nonumber M_k(N)&=M_{k-s}(N-s), s\leq k\leq N, 0\leq s\leq N.
\end{align}
Thus, we obtain that
\begin{align}
\nonumber \lim_{N\to\infty}\Psi_k(N)=\lim_{N\to\infty}\Psi_0(N-k)=\Psi.
\end{align}
Hence, $\Psi_k(N)$ is convergent. Now we shall show the convergence of $Z_k(N)$ and $X_k(N)$.

Since $Z_k(N)$ and $X_k(N)$ are uncorrelated with the initial value $x_0$, and the additive noise is not considered in this section, we set the initial value $x_0$ known, which means $\hat{x}_{0|0}=x_0$ and $\tilde{x}_0=0$. Obviously, the optimal cost function (\ref{34}) can be written as
\begin{align}
J^*_N=E[x_0'Z_0(N)x_0]. \label{82}
\end{align}
The proof of the convergence of $Z_k(N)$ is similar to the convergence of $\Psi_k(N)$, and it is omitted here. Thus, the convergence of $Z_k(N)$ is obtained. Recalling (\ref{67}), due to the convergence of $Z_k(N)$ and $\Psi_k(N)$, it can be obtained that $X_k(N)$ is convergent.

Next we will prove that there exists $N_0>0$ such that $\Psi_0(N_0)>0$. Suppose this is not the case. Then there exists nonzero $x$ satisfying $E[x'\Psi_0(N)x]=0$. Let the initial value $x_0=x$. Then the optimal cost function of (\ref{48}) is as
\begin{align}
\nonumber J^*_N&=\sum^N_{k=0}E\bigg[x_k^{*'}Qx_k^{*}+(u^{*R}_k)^{'}R^Ru^{*R}_k+(u^{*L}_k)'R^Lu^{*L}_k\bigg]\\
\nonumber      &=E[x'\Psi_0(N)x]\\
\nonumber      &=0,
\end{align}
where $x_k^{*}$ presents the optimal state trajectory, $u^{*R}_k$ and $u^{*L}_k$ stand for the optimal controllers, respectively. Noting Assumption 1, i.e., $R^L>0$, $R^R>0$ and $Q=C'C\geq0$, it follows that
\begin{align}
\nonumber u^{*R}_k=0, u^{*L}_k=0, Cx_k^*=0, 0\leq k\leq N, N\geq 0.
\end{align}
Note Assumption 2, i.e., $(A,Q^{1/2})$ is observable. It can be obtained that $x_0=x=0$, which is a contradiction to $x\neq0$. Thus, there exists $N_0\geq0$ such that $\Psi_0(N_0)>0$. Therefore, $\Psi=\lim_{N\to\infty}\Psi_0(N)>0$ has been shown. By setting the initial value $x_0$ known, noting (\ref{82}), the proof of $Z>0$ is similar to the proof of $\Psi>0$ and it is omitted here.

Next, we will show the uniqueness of the solutions $Z$ and $X$ to (\ref{24})-(\ref{36}). Let $Z^e$ and $X^e$, be other solutions to (\ref{24})-(\ref{36}) satisfying $Z^e>0$ and $\Psi^e>0$, i.e.,
\begin{align}
\nonumber Z^e&=A'Z^eA+Q-K^{e'}\Upsilon K^e,\\
\nonumber X^e&=(1-p)A'Z^eA\hspace{-0.8mm}+pA'X^eA\hspace{-0.8mm}+\hspace{-0.8mm}Q\hspace{-0.8mm}-\hspace{-0.8mm}M^{e'}\Lambda^{e^{-1}}M^e,
\end{align}
where
\begin{align}
\nonumber  K^e&=\Upsilon^{e^{-1}}\begin{bmatrix}B^L&B^R\end{bmatrix}'Z^eA,\\
\nonumber  \Upsilon^e&=\begin{bmatrix}B^L&B^R\end{bmatrix}'Z^e\begin{bmatrix}B^L&B^R\end{bmatrix}+\begin{bmatrix}R^L&0\\0&R^R\end{bmatrix},\\
\nonumber \Lambda&=(1-p)(B^L)'Z^eB^L\hspace{-0.8mm}+p(B^L)'X^eB^L\hspace{-0.8mm}+R^L,\\
          \Psi^e&=(1-p)Z^e+pX^e,\label{59}\\
\nonumber M^e&=(1-p)(B^L)'Z^eA+p(B^L)'X^eA.
\end{align}
Recalling (\ref{48}), the optimal value of the cost function is as
\begin{align}
\nonumber J^*=E[x_0'\Psi^e x_0]=E[x_0'\Psi x_0].
\end{align}
As $x_0$ is arbitrary, the above equation indicates that
\begin{align}
\Psi^e=\Psi.\label{49}
\end{align}
Via setting the initial value $x_0$ known, with (\ref{82}), the uniqueness of $Z$ can be obtained. Combining (\ref{67}) and (\ref{49}), it yields that $X$ is unique.

At last, we shall show that $E[{x}_{0}'Z\hat{x}_{0|0}+{x}_0'X\tilde{x}_0]\geq0$ for any initial value $x_0$. Since the additive noise is not considered in this section, the optimal cost (\ref{34}) can be written as
\begin{align}
\nonumber J^*_N&=E\big[x_0'Z_0(N)\hat{x}_{0|0}+x_0'X_0(N)\tilde{x}_0\big]
\end{align}
Since $J^*_N\geq$ always holds for any initial value $x_0$, we have that $E[{x}_{0}'Z\hat{x}_{0|0}+{x}_0'X\tilde{x}_0]\geq0$ for any initial value $x_0$. In other word, if the initial time is $n$, $n\in N$, we have that $E[{x}_{n}'Z\hat{x}_{n|n}+{x}_n'X\tilde{x}_n]\geq0$. This completes the proof of the necessity.

``Sufficiency": Under Assumptions 1 and 2, suppose that $Z$ and $X$ are the solutions to (\ref{24}) and (\ref{26}) satisfy $Z>0$, $\Psi>0$ and $E[{x}_{0}'Z\hat{x}_{0|0}+{x}_0'X\tilde{x}_0]\geq0$ for any initial value $x_0$. We shall show that (\ref{37}) and (\ref{38}) stabilize system (\ref{16}) in the mean square sense.

Using (\ref{b}), system (\ref{16}) can be written as
\begin{align}
x_{k+1}=Ax_{k}+\begin{bmatrix}B^L&B^R\end{bmatrix}\begin{bmatrix}\hat{u}^L_k\\u^R_k\end{bmatrix}+B^L\tilde{u}^L_k. \label{I}
\end{align}

Define the Lyapunov function candidate $V(k,x_k)$ as
\begin{align}
V(k,x_k)=E\left[x_k'Z\hat{x}_{k|k}+x_k'X\tilde{x}_k\right].\label{40}
\end{align}
Next we shall show the convergence of $V(k,x_k)$. Using (\ref{I}), (\ref{24})-(\ref{36}) and (\ref{c}), it yields that
\begin{align}
\nonumber&V(k,x_k)-V(k+1,x_{k+1})\\
\nonumber&=E\bigg\{x_k'(Z-A'ZA+K'\Upsilon K)x_k\\
\nonumber&\quad-\begin{bmatrix}\hat{u}^L_k\\u^R_k\end{bmatrix}'\bigg(\Upsilon-\hspace{-0.8mm}\begin{bmatrix}R^L\hspace{-0.8mm}&\hspace{-0.8mm}0\\0\hspace{-0.8mm}&\hspace{-0.8mm}R^R\end{bmatrix}\bigg)\begin{bmatrix}\hat{u}^L_k\\u^R_k\end{bmatrix}-\hspace{-0.8mm}2\begin{bmatrix}\hat{u}^L_k\\u^R_k\end{bmatrix}'\Upsilon K\hat{x}_{k|k}\\
\nonumber&\quad-\hat{x}_{k|k}'K'\Upsilon K\hat{x}_{k|k}\hspace{-0.8mm}-(u^L_k)'(\Lambda\hspace{-0.8mm}-R^L)u^L_k\hspace{-0.8mm}-2(u^L_k)'M\tilde{x}_k\\
\nonumber&\quad-\tilde{x}_k'(pA'XA\hspace{-0.8mm}-\hspace{-0.8mm}pA'ZA\hspace{-0.8mm}-\hspace{-0.8mm}X\hspace{-0.8mm}+\hspace{-0.8mm}Z\hspace{-0.8mm}+\hspace{-0.8mm}K'\Upsilon K)\tilde{x}_k\bigg\}\\
\nonumber&=E\bigg\{x_k'Qx_k+\begin{bmatrix}\hat{u}^L_k\\u^R_k\end{bmatrix}'\begin{bmatrix}R^L\hspace{-0.8mm}&\hspace{-0.8mm}0\\0\hspace{-0.8mm}&\hspace{-0.8mm}R^R\end{bmatrix}\begin{bmatrix}\hat{u}^L_k\\u^R_k\end{bmatrix}+(\tilde{u}^L_k)'R^L\tilde{u}^L_k\\
\nonumber&\quad-\left(\begin{bmatrix}\hat{u}^L_k\\u^R_k\end{bmatrix}+K\hat{x}_{k|k}\right)'\Upsilon\left(\begin{bmatrix}\hat{u}^L_k\\u^R_k\end{bmatrix}+K\hat{x}_{k|k}\right)\\
         &\quad-(\tilde{u}^L_k+\Lambda^{-1}M\tilde{x}_k)'\Lambda(\tilde{u}^L_k+\Lambda^{-1}M\tilde{x}_k)\bigg\}\label{52}\\
         &=E\bigg\{\hspace{-0.8mm}x_k'Qx_k\hspace{-0.8mm}+\hspace{-0.8mm}\begin{bmatrix}\hat{u}^L_k\\u^R_k\end{bmatrix}'\begin{bmatrix}R^L\hspace{-0.8mm}&\hspace{-0.8mm}0\\0\hspace{-0.8mm}&\hspace{-0.8mm}R^R\end{bmatrix}\hspace{-0.8mm}\begin{bmatrix}\hat{u}^L_k\\u^R_k\end{bmatrix}\hspace{-0.8mm}+\hspace{-0.8mm}(\tilde{u}^L_k)'R^L\tilde{u}^L_k\bigg\}\hspace{-0.8mm}\geq0,\label{39}
\end{align}
where (\ref{37}) and (\ref{38}) have been applied in the last identity. Obviously, $V(k,x_k)$ is monotonically decreasing with respect to $k$. Since $E[{x}_{0}'Z\hat{x}_{0|0}+{x}_0'X\tilde{x}_0]\geq0$ for any initial value $x_0$, we have that $V(k,x_k)\geq0$, i.e., $V(k,x_k)$ is bounded below. Thus $V(k,x_k)$ is convergent.

Now let $l$ be any nonnegative integer. Taking summation from $k=l$ to $k=l+N$ on both side of (\ref{39}), and letting $l\to\infty$, it yields that
\begin{align}
\nonumber&\lim_{l\to\infty}\hspace{-0.8mm}\sum_{k=l}^{l+N}\hspace{-0.8mm}E\bigg\{\hspace{-0.8mm}x_k'Qx_k\hspace{-0.8mm}+\hspace{-0.8mm}\begin{bmatrix}\hat{u}^L_k\\u^R_k\end{bmatrix}'\begin{bmatrix}R^L\hspace{-0.8mm}&\hspace{-0.8mm}0\\0\hspace{-0.8mm}&\hspace{-0.8mm}R^R\end{bmatrix}\hspace{-0.8mm}\begin{bmatrix}\hat{u}^L_k\\u^R_k\end{bmatrix}\hspace{-0.8mm}+\hspace{-0.8mm}(\tilde{u}^L_k)'R^L\tilde{u}^L_k\bigg\}\\
         &=\lim_{l\to\infty}\left[V(m,x_m)-V(m+N+1,x_{m+N+1})\right]=0,\label{41}
\end{align}
where the convergence of $V(k,x_k)$ is imposed in the last identity.

Noting (\ref{25}) and letting $\bar{x}_0=0$, we have the cost function as
\begin{align}
\nonumber&\sum_{k=0}^{N}E\bigg[x_k'Qx_k\hspace{-0.8mm}+\hspace{-0.8mm}\begin{bmatrix}\hat{u}^L_k\\u^R_k\end{bmatrix}'\begin{bmatrix}R^L&0\\0&R^R\end{bmatrix}\begin{bmatrix}\hat{u}^L_k\\u^R_k\end{bmatrix}\hspace{-0.8mm}+\hspace{-0.8mm}(\tilde{u}^L_k)'R^L\tilde{u}^L_k\bigg]\\
\nonumber&\geq E\big[x_0'Z_0\hat{x}_{0|0}\hspace{-0.8mm}+\hspace{-0.8mm}x_0'X_0\tilde{x}_0\big]=E\left\{x_0'\Phi_0x_0\right\}.
\end{align}
Through a time-shift of length of $l$, it yields that
\begin{align}
\nonumber&\sum_{k=l}^{l+N}\hspace{-0.8mm}E\bigg[x_k'Qx_k\hspace{-0.8mm}+\hspace{-0.8mm}\begin{bmatrix}\hat{u}^L_k\\u^R_k\end{bmatrix}'\begin{bmatrix}R^L&0\\0&R^R\end{bmatrix}\begin{bmatrix}\hat{u}^L_k\\u^R_k\end{bmatrix}\hspace{-0.8mm}+\hspace{-0.8mm}(\tilde{u}^L_k)'R^L\tilde{u}^L_k\bigg]\\
&\geq E\left\{x_l'\Phi_lx_l\right\}\geq0.\label{42}
\end{align}
Combining with (\ref{41}) and (\ref{42}), it follows that
\begin{align}
\lim_{l\to\infty}E\left[x_{l}'x_{l}\right]=0.
\end{align}
Therefore, the closed-loop system (\ref{16}) is stabilizable in the mean-square sense by the controllers (\ref{37}) and (\ref{38}).

Next we shall show that (\ref{37}) and (\ref{38}) minimize the cost function (\ref{21}). Taking summation from $k=0$ to $k=N$ on both sides of (\ref{52}), we have
\begin{align}
\nonumber&E\sum_{k=0}^N\bigg[{x_k}'Qx_k\hspace{-0.8mm}+\hspace{-0.8mm}\begin{bmatrix}\hat{u}^L_k\\u^R_k\end{bmatrix}'\begin{bmatrix}R^L&0\\0&R^R\end{bmatrix}\begin{bmatrix}\hat{u}^L_k\\u^R_k\end{bmatrix}\hspace{-0.8mm}+\hspace{-0.8mm}(\tilde{u}^L_k)'R^L\tilde{u}^L_k\bigg]\\
\nonumber &=V(0,x_0)-V(N+1,x_{N+1})\\
\nonumber&\quad+E\bigg\{\left(\begin{bmatrix}\hat{u}^L_k\\u^R_k\end{bmatrix}+K\hat{x}_{k|k}\right)'\Upsilon\left(\begin{bmatrix}\hat{u}^L_k\\u^R_k\end{bmatrix}+K\hat{x}_{k|k}\right)\\
         &\quad+(\tilde{u}^L_k+\Lambda^{-1}M\tilde{x}_k)'\Lambda(\tilde{u}^L_k+\Lambda^{-1}M\tilde{x}_k)\bigg\},\label{43}
\end{align}
where $V(0,x_0)$ and $V(N+1,x_{N+1})$ are defined in (\ref{40}). Combining with (\ref{50}) and (\ref{40}), it yields that
\begin{align}
\nonumber \lim_{k\to\infty}V(k,x_k)=0.
\end{align}
Thus, by letting $N\to\infty$ on both sides of (\ref{43}), the cost function (\ref{21}) is rewritten as
\begin{align}
\nonumber J&=E\left[x_0'Z\hat{x}_{0|0}+x_0'X\tilde{x}_0\right]\\
\nonumber&\quad+E\bigg\{\left(\begin{bmatrix}\hat{u}^L_k\\u^R_k\end{bmatrix}+K\hat{x}_{k|k}\right)'\Upsilon\left(\begin{bmatrix}\hat{u}^L_k\\u^R_k\end{bmatrix}+K\hat{x}_{k|k}\right)\\
         &\quad+(\tilde{u}^L_k+\Lambda^{-1}M\tilde{x}_k)'\Lambda(\tilde{u}^L_k+\Lambda^{-1}M\tilde{x}_k)\bigg\}.\label{53}
\end{align}
Due to the positive definiteness of $\Upsilon$ and $\Lambda$, the optimal controllers to minimize (\ref{53}) must be (\ref{37}) and (\ref{38}). Moreover, the corresponding optimal cost is as (\ref{54}). The proof of sufficiency is completed.
\end{proof}
\section{Proof of Lemma 2}
\begin{proof}
Using (\ref{l}), we have that
\begin{align}
\nonumber \tilde{x}_k=(1-\eta_{k})(A\tilde{x}_{k-1}+B^L\tilde{u}^L_{k-1}+\omega_k).
\end{align}
Noting the assumption that system (\ref{16}) is stabilizable in the mean square sense. Thus, the algebraic Riccati equations (\ref{24}) and (\ref{26}) hold. Substituting (\ref{57}) into the above equation, it yields that
\begin{align}
\nonumber\tilde{x}_k&=(1-\eta_{k})[(A-B^L\Lambda^{-1}M)\tilde{x}_{k-1}+\omega_k].
\end{align}
Thus, the estimator error covariance matrix
\begin{align}
\nonumber\Sigma_k&=p(A-B^L\Lambda^{-1}M)\Sigma_{k-1}(A-B^L\Lambda^{-1}M)'+pQ_{\omega}\\
\nonumber        &=p^2(A-B^L\Lambda^{-1}M)^2\Sigma_{k-2}[(A-B^L\Lambda^{-1}M)']^2\\
\nonumber        &\quad+p^2(A\hspace{-0.8mm}-\hspace{-0.8mm}B^L\Lambda^{-1}M)Q_{\omega}(A\hspace{-0.8mm}-\hspace{-0.8mm}B^L\Lambda^{-1}M)'\hspace{-0.8mm}+\hspace{-0.8mm}pQ_{\omega}\\
\nonumber        &\qquad\vdots\\
\nonumber        &=p^k(A-B^L\Lambda^{-1}M)^k\Sigma_0[(A-B^L\Lambda^{-1}M)']^k\\
\nonumber        &\quad+\hspace{-0.8mm}\sum_{i=1}^kp^i(A\hspace{-0.8mm}-\hspace{-0.8mm}B^L\Lambda^{-1}M)^{i-1}Q_{\omega}[(A\hspace{-0.8mm}-B^L\Lambda^{-1}M)']^{i-1}
\end{align}
``Sufficiency'': If $\sqrt{p}|\lambda_{max}(A-B^L\Lambda^{-1}M)|<1$, we shall show that $\Sigma_k$ is convergent.

Obviously, since $\sqrt{p}|\lambda_{max}(A-B^L\Lambda^{-1}M)|<1$, we have $\lim_{k\to\infty}p^k(A-B^L\Lambda^{-1}M)^k\Sigma_0[(A-B^L\Lambda^{-1}M)']^k=0$. Accordingly, we have that
\begin{align}
\nonumber\lim_{k\to\infty}\hspace{-0.8mm}\sum_{i=1}^kp^i(A\hspace{-0.8mm}-\hspace{-0.8mm}B^L\Lambda^{-1}M)^{i-\hspace{-0.8mm}1}Q_{\omega}[(A\hspace{-0.8mm}-\hspace{-0.8mm}B^L\Lambda^{-1}M)']^{i-\hspace{-0.8mm}1}\hspace{-0.8mm}=\hspace{-0.8mm}P,
\end{align}
where $P$ is the unique solution of the following equation:
\begin{align}
\nonumber P=p(A-B^L\Lambda^{-1}M)P(A-B^L\Lambda^{-1}M)'+pQ_{\omega}.
\end{align}
This completes the proof of the sufficiency.

``Necessity'': Suppose that $\Sigma_k$ is convergent, we shall show that $\sqrt{p}|\lambda_{max}(A-B^L\Lambda^{-1}M)|<1$. Assume that $\sqrt{p}|\lambda_{max}(A-B^L\Lambda^{-1}M)|\geq1$, it is readily obtained that
\begin{align}
\nonumber &\lim_{k\to\infty}p^k(A-B^L\Lambda^{-1}M)^k\Sigma_0[(A-B^L\Lambda^{-1}M)']^k=\infty,\\
\nonumber &\lim_{k\to\infty}\hspace{-0.8mm}\sum_{i=1}^kp^i(\hspace{-0.8mm}A\hspace{-0.8mm}-\hspace{-0.8mm}B^L\Lambda^{-1}M\hspace{-0.8mm})^{i-1}Q_{\omega}[(\hspace{-0.8mm}A\hspace{-0.8mm}-\hspace{-0.8mm}B^L\Lambda^{-1}M\hspace{-0.8mm})']^{i-1}\hspace{-0.8mm}=\hspace{-0.8mm}\infty,
\end{align}
which is contradicted to the convergence of $\Sigma_k$. Thus, we have $\sqrt{p}|\lambda_{max}(A-B^L\Lambda^{-1}M)|<1$. This completes the proof of the necessity.
\end{proof}
\section{Proof of Corollary 1}
\begin{proof}
Under Assumptions 1 and 2, suppose $\sqrt{p}|\lambda_{max}(A-B^L\Lambda^{-1}M)|<1$, and there exist the unique solutions $Z$ and $X$ to the Riccati equations (\ref{24})-(\ref{36}). We shall show that (\ref{56}) and (\ref{57}) make the system (\ref{1}) bounded in the mean-square sense.

Substituting (\ref{56}) and (\ref{57}) into the system (\ref{1}), it yields that
\begin{align}
\nonumber x_{k+1}&=Ax_k-\begin{bmatrix}B^L&B^R\end{bmatrix}K\hat{x}_{k|k}-B^L\Lambda^{-1}M\tilde{x}_k+\omega_k\\
\nonumber        &=\left(A-\begin{bmatrix}B^L&B^R\end{bmatrix}K\right)x_k\\
\nonumber        &\quad+\left(\begin{bmatrix}B^L&B^R\end{bmatrix}K-B^L\Lambda^{-1}M\right)\tilde{x}_k+\omega_k
\end{align}
From Lemma 2, the estimation error covariance $\lim_{k\to\infty}E\left[\tilde{x}_k\tilde{x}_k'\right]$ is bounded which implies that \\ $\lim_{k\to\infty}E\left[\tilde{x}_k'\tilde{x}_k\right]$ is bounded. Besides, the covariance of the additive noise $\omega_k$ is $Q_{\omega}$. Hence, $\lim_{k\to\infty}E\left[{x}_k'{x}_k\right]$ is bounded if and only if the linear system
\begin{align}
\beta_{k+1}=\left(A-\begin{bmatrix}B^L&B^R\end{bmatrix}K\right)\beta_k,\label{61}
\end{align}
with the initial value $\beta_0=x_0$, is stable in the mean square sense.

Noting (\ref{24}), (\ref{35}) and (\ref{60}), we rewrite (\ref{24}) as
\begin{align}
\nonumber Z&=K'\begin{bmatrix}R^L\hspace{-0.8mm}&\hspace{-0.8mm}0\\0\hspace{-0.8mm}&\hspace{-0.8mm}R^R\end{bmatrix}K+(A-\begin{bmatrix}B^L&B^R\end{bmatrix}K)'Z\\
           &\quad\times(A-\begin{bmatrix}B^L&B^R\end{bmatrix}K)+Q.\label{62}
\end{align}
Now we will show the system (\ref{61}) is stable in the mean square sense. Define the Lyapunov function candidate $W(k,\beta_k)$ as
\begin{align}
\nonumber W(k,\beta_k)=E\left[\beta_k'Z\beta_k\right].
\end{align}
Using (\ref{62}), it yields that
\begin{align}
\nonumber&W(k+1,\beta_{k+1})-W(k,\beta_k)\\
\nonumber&=\hspace{-0.8mm}E\big\{\beta_k'\big[\hspace{-0.8mm}\left(A\hspace{-0.8mm}-\hspace{-0.8mm}\begin{bmatrix}B^L\hspace{-0.8mm}&\hspace{-0.8mm}B^R\end{bmatrix}K\right)'\hspace{-0.8mm}Z\hspace{-0.8mm}\left(A\hspace{-0.8mm}-\hspace{-0.8mm}\begin{bmatrix}B^L\hspace{-0.8mm}&\hspace{-0.8mm}B^R\end{bmatrix}K\right)\hspace{-0.8mm}-Z\big]\beta_k\big\}\\
         &=-E\left[\beta_k'\left((K'\begin{bmatrix}R^L\hspace{-0.8mm}&\hspace{-0.8mm}0\\0\hspace{-0.8mm}&\hspace{-0.8mm}R^R\end{bmatrix}K\hspace{-0.8mm}+\hspace{-0.8mm}Q\right)\beta_k\right].\label{63}
\end{align}
Hence, $W(k,\beta_{k})$ decreases with respect to $k$ and bounded below, i.e., $W(k,\beta_{k})$ is convergent. Taking summation from $k=0$ to $k=l$ on both sides of (\ref{63}), we have
\begin{align}
\nonumber&W(l+1,\beta_{l+1})-W(0,\beta_0)\\
\nonumber &=-\sum_{k=0}^{l}E\left[\beta_k'\left(K'\begin{bmatrix}R^L\hspace{-0.8mm}&\hspace{-0.8mm}0\\0\hspace{-0.8mm}&\hspace{-0.8mm}R^R\end{bmatrix}K+Q\right)\beta_k\right].
\end{align}
By letting $l\to\infty$ on both sides of the above equation, it yields that
\begin{align}
\nonumber &\lim_{l\to\infty} E\left[\beta_{l+1}'Z\beta_{l+1}\right]\\
\nonumber&=E\left[\beta_0'Z\beta_0\right]\hspace{-0.8mm}-\hspace{-0.8mm}\lim_{l\to\infty}\sum_{k=0}^{l}\hspace{-0.8mm}E\left[\beta_k'\hspace{-0.8mm}\left(\hspace{-0.8mm}K'\begin{bmatrix}R^L\hspace{-0.8mm}&\hspace{-0.8mm}0\\0\hspace{-0.8mm}&\hspace{-0.8mm}R^R\end{bmatrix}K\hspace{-0.8mm}+\hspace{-0.8mm}Q\hspace{-0.8mm}\right)\hspace{-0.8mm}\beta_k\right]\hspace{-0.8mm}.
\end{align}
Due to the convergence of $W(k,\beta_{k})$, we have
\begin{align}
\nonumber \lim_{l\to\infty}E\left[\beta_l'\left(K'\begin{bmatrix}R^L\hspace{-0.8mm}&\hspace{-0.8mm}0\\0\hspace{-0.8mm}&\hspace{-0.8mm}R^R\end{bmatrix}K\hspace{-0.8mm}+\hspace{-0.8mm}Q\right)\beta_l\right]=0.
\end{align}
Noting Assumptions 1 and 2, we have $\lim_{l\to\infty}E\left[\beta_l'\beta_l\right]=0$. Hence, we obtain that $\lim_{k\to\infty}E\left[x_k'x_k\right]$ is bounded, i.e., (\ref{56}) and (\ref{57}) make system (\ref{1}) bounded in the mean-square sense.

Next we will prove that (\ref{56}) and (\ref{57}) minimize the cost function (\ref{55}). Define
\begin{align}
\tilde{V}(k,x_k)=E\left[x_k'Z\hat{x}_{k|k}+x_k'X\tilde{x}_k\right].\label{65}
\end{align}
Similar to (\ref{33}), we have
\begin{align}
\nonumber&\tilde{V}(k,x_k)-\tilde{V}(k+1,x_{k+1})\\
\nonumber&=E\bigg\{x_k'Qx_k+\begin{bmatrix}\hat{u}^L_k\\u^R_k\end{bmatrix}'\begin{bmatrix}R^L\hspace{-0.8mm}&\hspace{-0.8mm}0\\0\hspace{-0.8mm}&\hspace{-0.8mm}R^R\end{bmatrix}\begin{bmatrix}\hat{u}^L_k\\u^R_k\end{bmatrix}+(\tilde{u}^L_k)'R^L\tilde{u}^L_k\\
\nonumber&\quad-\left(\begin{bmatrix}\hat{u}^L_k\\u^R_k\end{bmatrix}+K\hat{x}_{k|k}\right)'\Upsilon\left(\begin{bmatrix}\hat{u}^L_k\\u^R_k\end{bmatrix}+K\hat{x}_{k|k}\right)\\
\nonumber&\quad-(\tilde{u}^L_k+\Lambda^{-1}M\tilde{x}_k)'\Lambda(\tilde{u}^L_k+\Lambda^{-1}M\tilde{x}_k)\bigg\}\\
         &\quad-pTr(Q_{\omega}X)-(1-p)Tr(Q_{\omega}Z).\label{64}
\end{align}
Taking summation from $k=0$ to $k=N$ on both sides of (\ref{64}) leads to
\begin{align}
\nonumber&\sum_{k=0}^NE\bigg(x_k'Qx_k+\begin{bmatrix}\hat{u}^L_k\\u^R_k\end{bmatrix}'\begin{bmatrix}R^L\hspace{-0.8mm}&\hspace{-0.8mm}0\\0\hspace{-0.8mm}&\hspace{-0.8mm}R^R\end{bmatrix}\begin{bmatrix}\hat{u}^L_k\\u^R_k\end{bmatrix}+(\tilde{u}^L_k)'R^L\tilde{u}^L_k\bigg)\\
\nonumber&=\tilde{V}(0,x_0)-\tilde{V}(N+1,x_{N+1})\\
\nonumber&\quad+\sum_{k=0}^N\left(\begin{bmatrix}\hat{u}^L_k\\u^R_k\end{bmatrix}+K\hat{x}_{k|k}\right)'\Upsilon\left(\begin{bmatrix}\hat{u}^L_k\\u^R_k\end{bmatrix}+K\hat{x}_{k|k}\right)\\
\nonumber&\quad+\sum_{k=0}^N(\tilde{u}^L_k+\Lambda^{-1}M\tilde{x}_k)'\Lambda(\tilde{u}^L_k+\Lambda^{-1}M\tilde{x}_k)\bigg\}\\
         &\quad+\sum_{k=0}^N\left[pTr(Q_{\omega}X)+(1-p)Tr(Q_{\omega}Z)\right].\label{68}
\end{align}
Let $k\to\infty$ on both sides of (\ref{65}), we have
\begin{align}
\nonumber&\lim_{k\to\infty}\tilde{V}(k,x_k)\\
\nonumber&=\lim_{k\to\infty}E\left[(\hat{x}_{k|k}+\tilde{x}_k)'Z\hat{x}_{k|k}+(\hat{x}_{k|k}+\tilde{x}_k)'X\tilde{x}_k\right]\\
\nonumber&=\lim_{k\to\infty}E[\hat{x}_{k|k}'Z\hat{x}_{k|k}]+E[\tilde{x}_k'X\tilde{x}_k]
\end{align}
Since $\lim_{k\to\infty}E\left[x_k'x_k\right]$ and $\lim_{k\to\infty}E\left[\tilde{x}_k'\tilde{x}_k\right]$ are both bounded, it leads that $\lim_{k\to\infty}E[\hat{x}_{k|k}'\hat{x}_{k|k}]$ is bounded. Thus, $\lim_{k\to\infty}\tilde{V}(k,x_k)$ is bounded.

By letting $N\to\infty$ on both sides of (\ref{68}), the cost function (\ref{55}) is rewritten as
\begin{align}
\nonumber\tilde{J}&=\lim_{N\to\infty}\frac{1}{N}\bigg\{\tilde{V}(0,x_0)-\tilde{V}(N+1,x_{N+1})\\
\nonumber&\quad+\sum_{k=0}^N\left(\begin{bmatrix}\hat{u}^L_k\\u^R_k\end{bmatrix}+K\hat{x}_{k|k}\right)'\Upsilon\left(\begin{bmatrix}\hat{u}^L_k\\u^R_k\end{bmatrix}+K\hat{x}_{k|k}\right)\\
\nonumber&\quad+\sum_{k=0}^N(\tilde{u}^L_k+\Lambda^{-1}M\tilde{x}_k)'\Lambda(\tilde{u}^L_k+\Lambda^{-1}M\tilde{x}_k)\bigg\}\\
\nonumber&\quad+\sum_{k=0}^N\left[pTr(Q_{\omega_k}X_{k+1})\hspace{-0.8mm}+\hspace{-0.8mm}(1-p)Tr(Q_{\omega_k}Z_{k+1})\right]\bigg\}\\
\nonumber&=\left(\begin{bmatrix}\hat{u}^L_k\\u^R_k\end{bmatrix}+K\hat{x}_{k|k}\right)'\Upsilon\left(\begin{bmatrix}\hat{u}^L_k\\u^R_k\end{bmatrix}+K\hat{x}_{k|k}\right)\\
\nonumber&\quad+(\tilde{u}^L_k+\Lambda^{-1}M\tilde{x}_k)'\Lambda(\tilde{u}^L_k+\Lambda^{-1}M\tilde{x}_k)\bigg\}\\
         &\quad+\left[pTr(Q_{\omega}X)+(1-p)Tr(Q_{\omega}Z)\right].\label{66}
\end{align}
In view of the positive definiteness of $\Upsilon$ and $\Lambda$, the optimal controllers to minimize (\ref{66}) must be (\ref{56}) and (\ref{57}). Moreover, the corresponding optimal cost is as (\ref{55}). The proof of Corollary 1 is completed.
\end{proof}
\end{document}